\documentclass[12pt,reqno]{amsart}

\usepackage{amsfonts,color,amsthm,amsmath,amssymb}

\usepackage{color}
\allowdisplaybreaks[3]

\def\Box{\vcenter{\vbox{\hrule\hbox{\vrule
     \vbox to 8.8pt{\hbox to 10pt{}\vfill}\vrule}\hrule}}}

\newcommand{\tr}{\text{Tr}}
\def\qed{{\hfill$\square$}}
\def\proof{{\vspace{-0.0cm}\bf Proof: \,}}

\def\Z{{\mathbb Z}}

\def\F{{\mathbb F}}

\def\mod{{\mathrm{mod\,\,}}}

\def\Tr{{\mathrm{Tr}}}
\def\Norm{{\mathrm{Norm}}}

\def\PG{{\mathrm{PG}}}
\def\AG{{\mathrm{AG}}}

\newtheorem{theorem}{Theorem}[section]

\newtheorem{lemma}[theorem]{Lemma}
\newtheorem{remark}[theorem]{Remark}

\newtheorem{proposition}[theorem]{Proposition}

\newtheorem{conj}[theorem]{Conjecture}
\numberwithin{equation}{section}

\newtheorem{result}[theorem]{Result}
\newtheorem{notation}[theorem]{Notation}
\def\cQ{{\mathcal Q}}
\def\cM{{\mathcal M}}
\def\cK{{\mathcal K}}
\def\cP{{\mathcal P}}
\def\cL{{\mathcal L}}
\def\la{\langle}
\def\ra{\rangle}
\def\sgn{\textup{sgn}}

\begin{document}
\title[Cameron-Liebler line classes]{Cameron-Liebler line classes with parameter $x=\frac{q^2-1}{2}$}

\author[Feng, Momihara and Xiang]{Tao Feng$^*$, Koji Momihara$^{\dagger}$ and Qing Xiang}
\thanks{$^*$Research supported in part by Fundamental Research Fund for the Central Universities of China, Zhejiang University K.P. Chao's High Technology Development Foundation, the National Natural Science Foundation of China under Grant No. 11201418 and 11422112, and the Research Fund for Doctoral Programs from the Ministry of Education of China under Grant No. 20120101120089.}
\thanks{$^{\dagger}$
Research supported by JSPS under Grant-in-Aid for Young Scientists (B) 25800093 and Scientific Research (C) 24540013.}

\address{Department of Mathematics, Zhejiang University, Hangzhou 310027, Zhejiang, P. R. China}
\email{tfeng@zju.edu.cn}

\address{Faculty of Education, Kumamoto University, 2-40-1 Kurokami, Kumamoto 860-8555, Japan} \email{momihara@educ.kumamoto-u.ac.jp}

\address{Department of Mathematical Sciences, University of Delaware, Newark, DE 19716, USA} \email{xiang@math.udel.edu}

\begin{abstract}
In this paper, we give an algebraic construction of a new infinite family of Cameron-Liebler line classes with parameter $x=\frac{q^2-1}{2}$ for $q\equiv 5$ or $9\pmod{12}$, which generalizes the examples found by Rodgers in  \cite{rodgers} through a computer search.
Furthermore, in the case where  $q$ is an even power of $3$, we construct the first infinite family of affine two-intersection sets in $\AG(2,q)$, which is closely related to our Cameron-Liebler line classes.
\end{abstract}

\maketitle

\section{Introduction}
Cameron-Liebler line classes were first introduced by Cameron and Liebler \cite{CL} in their study of collineation groups of $\PG(3,q)$ having the same number of orbits on points and lines of $\PG(3,q)$. Later on it was found that these line classes have many connections to other geometric and combinatorial objects, such as blocking sets of $\PG(2,q)$, projective two-intersection sets in $\PG(5,q)$, two-weight linear codes, and strongly regular graphs.  In the last few years, Cameron-Liebler line classes have received considerable  attention from researchers in both finite geometry and algebraic combinatorics; see, for example, \cite{DHS, Metsch1, Metsch2, rodgers, Gav, GavM}. In \cite{CL}, the authors gave several equivalent conditions for a set of lines of $\PG(3,q)$ to be a Cameron-Liebler line class; Penttila \cite{Pchar} gave a few more of such characterizations. We will use one of these characterizations as the definition of a Cameron-Liebler line class. Let $\cL$ be a set of lines of $\PG(3,q)$ with $|\cL|=x(q^2+q+1)$, $x$ a nonnegative integer. We say that $\cL$ is a {\it Cameron-Liebler line class with parameter} $x$ if every spread of $\PG(3,q)$ contains $x$ lines of $\cL$. Clearly the complement of a Cameron-Liebler line class with parameter $x$ in the set of all lines of $\PG(3,q)$ is a Cameron-Liebler line class with parameter $q^2+1-x$. So without loss of generality we may assume that $ x\leq \frac{q^2+1}{2}$ when discussing Cameron-Liebler line classes of parameter $x$.  

Let $(P,\pi)$ be any non-incident point-plane pair of $\PG(3,q)$. Following \cite{Pchar}, we define $star(P)$ to be the set of all lines through $P$, and $line(\pi)$ to be the set of all lines contained in the plane $\pi$. We have the following trivial examples:
\begin{enumerate}
\item The empty set gives a Cameron-Liebler line class with parameter $x=0$;
\item Each of $star(P)$ and $line(\pi)$ gives a  Cameron-Liebler line class with parameter $x=1$;
\item $star(P)\cup line(\pi)$ gives a Cameron-Liebler line class with parameter $x=2$.
\end{enumerate}
Cameron-Liebler line classes are rare.  It was once conjectured (\cite[p.~97]{CL}) that the above trivial examples and their complements are all of the Cameron-Liebler line classes. The first counterexample to this conjecture was given by Drudge \cite{drudge} in $\PG(3,3)$, and it has parameter $x=5$. Later Bruen and Drudge \cite{BD} generalized Drudge's example into an infinite family with parameter $x=\frac{q^2+1}{2}$ for all odd $q$. This represents the only known infinite family of nontrivial Cameron-Liebler line classes before our work. Govaerts and Penttila \cite{GP} gave a sporadic example with parameter $x=7$ in $\PG(3,4)$. Recent work by Rodgers suggests that there are probably more infinite families of Cameron-Liebler line classes awaiting to be discovered. In \cite{rodgers}, Rodgers obtained new Cameron-Liebler line classes with parameter $x=\frac{q^2-1}{2}$ for $q\equiv 5$ or $9\pmod{12}$ and $q<200$. In his thesis \cite{Rthesis}, Rodgers also reported new examples with parameters $x=\frac{(q+1)^2}{3}$ for  $q\equiv 2\pmod{3}$ and $q<150$ as joint work with his collaborators.
These examples motivated us to find new general constructions of Cameron-Liebler line classes.

On the nonexistence side, Govaerts and Storme \cite{GS} first showed that there are no Cameron-Liebler line classes in $\PG(3,q)$ with parameter $2<x\leq q$ when $q$ is prime. Then De Beule, Hallez and Storme \cite{DHS} excluded parameters $2<x\leq q/2$ for all values $q$. Next Metsch \cite{Metsch1} proved the non-existence of Cameron-Liebler line classes with parameter $2<x\leq q$, and subsequently improved this result by showing the nonexistence of Cameron-Liebler line classes with parameter $2<x<q\sqrt[3]{\frac{q}{2}}-\frac{2}{3}q$ \cite{Metsch2}. The latter result represents the best asymptotic nonexistence result to date.  It seems reasonable to believe that for any fixed $0<\epsilon<1$ and constant $c>0$ there are no Cameron-Liebler line classes with $2<x<cq^{2-\epsilon}$ for sufficiently large $q$. Very recently, Gavrilyuk and Metsch \cite{GavM} proved a modular equality which eliminates almost half of the possible values $x$ for a Cameron-Liebler line class with parameter $x$. We refer to \cite{Metsch2} for a comprehensive survey of the known nonexistence results.

In the present paper we construct a new infinite family of Cameron-Liebler line classes with parameter $x=\frac{q^2-1}{2}$ for $q\equiv 5$ or $9\pmod{12}$. This family of Cameron-Liebler line classes generalizes the examples found by Rodgers in \cite{rodgers} through a computer search.  Furthermore, in the case where  $q$ is an even power of $3$, we construct the first infinite family of affine two-intersection sets,  which is closely related to the newly constructed Cameron-Liebler line classes. The first step of our construction follows the same idea as in \cite{rodgers}. That is, we prescribe an automorphism group for the Cameron-Liebler line classes that we intend to construct; as a consequence, the Cameron-Liebler line classes will be unions of orbits of the prescribed automorphism group on the set of lines of $\PG(3,q)$. The main difficulty with this approach is how to choose orbits properly so that their union is a Cameron-Liebler line class. We overcome this difficulty by giving an explicit choice of orbits so that their union gives a Cameron-Liebler line class with the required parameters. The details are given in Section 4.

The paper is organized as follows.  In Section~\ref{sec:pre}, we review basic properties of and facts on Cameron-Liebler line classes;  furthermore,  we collect auxiliary results on characters of finite fields, which are needed in the proof of our main theorem.  In Section~\ref{sec:setX}, we introduce a subset of $\F_{q^3}$, which we will use in the construction of our Cameron-Liebler line classes, and prove a few properties of the subset. In Section~\ref{proofPDS}, we give an algebraic construction of an infinite family of Cameron-Liebler line classes with $x=\frac{q^2-1}{2}$ for $q\equiv 5$ or $9\pmod{12}$. In Section~\ref{sec:aff}, we construct the first infinite family of affine two-intersection sets in $\AG(2,q)$, $q$ odd, whose existence was conjectured in the thesis \cite{Rthesis} of Rodgers. We close the paper with some concluding remarks.

\section{Preliminaries}\label{sec:pre}
In this section, we review basic facts on Cameron-Liebler line classes,  and collect auxiliary results on characters of finite fields.
\subsection{Preliminaries on Cameron-Liebler line classes}
It is often advantageous to study Cameron-Liebler line classes in $\PG(3,q)$ by using their images under the Klein correspondence. Let $\cQ^+(5,q)$ be the 5-dimensional hyperbolic orthogonal space and $x$ be a nonnegative integer. A subset $\mathcal{M}$ of $\cQ^+(5,q)$ is called an {\it $x$-tight set} if for every point $P\in \cQ^+(5,q)$, $|P^\perp\cap\mathcal{M}|=x(q+1)+q^2$ or $x(q+1)$ according as $P$ is in $\mathcal{M}$ or not, where $\perp$ is the polarity determined by $\cQ^+(5,q)$. The geometries of $\PG(3,q)$ and $\cQ^+(5,q)$ are closely related through a mapping known as the Klein correspondence which maps the lines of $\PG(3,q)$ bijectively to  the points of $\cQ^+(5,q)$, c.f. \cite{h2,stan}. Let $\cL$ be a set of lines of $\PG(3,q)$ with $|\cL|=x(q^2+q+1)$, $x$ a nonnegative integer, and let $\cM$ be the image of $\cL$ under the Klein correspondence. Then it is known that $\cL$ is a Cameron-Liebler line class with parameter $x$ in $\PG(3,q)$ if and only if $\cM$ is an $x$-tight set of  $\cQ^+(5,q)$.  Moreover, if $\cL$ is a Cameron-Liebler line class with parameter $x$, by \cite[Theorem 2.1 (b)]{Metsch1}, it holds that $|P^\perp\cap\mathcal{M}|=x(q+1)$ for any point $P$ off $\cQ^+(5,q)$; consequently $\mathcal{M}$ is a projective two-intersection set in $\PG(5,q)$ with intersection sizes $h_1=x(q+1)+q^2$ and $h_2=x(q+1)$, namely each hyperplane of $\PG(5,q)$ intersects $\mathcal{M}$ in either $h_1$ or $h_2$ points.  We summarize these known facts as follows.

\begin{result}\label{res_1} Let $\mathcal{L}$ be a set of $x(q^2+q+1)$ lines in $\PG(3,q)$, with $0<x\leq \frac{q^2+1}{2}$, and let $\cM$ be the image of $\cL$ under the Klein correspondence. Then $\mathcal{L}$ is a Cameron-Liebler line class with parameter $x$ if and only if $\mathcal{M}$ is an $x$-tight set in $\cQ^+(5,q)$; moreover, in the case when $\cL$ is a Cameron-Liebler line class, we have
\begin{align*}
|P^\perp\cap\mathcal{M}|=\begin{cases}x(q+1)+q^2,\quad &\textup{if $P\in\mathcal{M}$},\\
x(q+1),\quad & \textup{otherwise}.
\end{cases}
\end{align*}
\end{result}

A $(v,k,\lambda,\mu)$ {\it strongly regular graph} is a simple undirected regular graph on $v$ vertices with valency $k$ satisfying the following: for any two adjacent (resp. nonadjacent) vertices $x$ and $y$ there are exactly $\lambda$ (resp. $\mu$) vertices adjacent to both $x$ and $y$. It is known that a graph with valency $k$, not complete or edgeless, is strongly regular if and only if its adjacency matrix has exactly two restricted eigenvalues. Here, we say that an eigenvalue of the adjacency matrix is {\it restricted} if it has an eigenvector perpendicular to the all-ones vector.

One of the most effective methods for constructing strongly regular graphs is by the Cayley graph construction. Let $G$ be a finite abelian group and $D$ be an inverse-closed subset of $G\setminus\{0\}$.  We define a graph ${\rm Cay}(G,D)$ with the elements of $G$ as its vertices; two vertices $x$ and $y$ are adjacent if and only if $x-y\in D$. The graph ${\rm Cay}(G,D)$ is called a {\it Cayley graph} on $G$ with connection set $D$. The eigenvalues of ${\rm Cay}(G,D)$ are given by $\psi(D)$, $\psi\in {\widehat G}$, where ${\widehat G}$ is the group consisting of all complex characters of $G$, c.f.~\cite[\S1.4.9]{bh}. Using the aforementioned spectral characterization of strongly regular graphs, we see that ${\rm Cay}(G,D)$ with connection set $D$($\not=\emptyset,G$) is strongly regular if and only if $\psi(D)$, $\psi\in {\widehat G}\setminus\{1\}$, take exactly two values, say $\alpha_1$ and $\alpha_2$ with $\alpha_1>\alpha_2$. We note that if ${\rm Cay}(G,D)$ is strongly regular with two restricted eigenvalues $\alpha_1$ and $\alpha_2$, then the set
$\{\psi\in {\widehat G}\,|\,\psi(D)=\alpha_1\}$ also forms a connection set of a strongly regular Cayley graph on ${\widehat G}$; this set is called the {\it dual} of $D$. For basic properties of strongly regular graphs, see \cite[Chapter 9]{bh}. For known constructions of  strongly regular Cayley graphs and their connections to two-weight linear codes, partial difference sets, and finite geometry, see \cite[p.~133]{bh} and \cite{CK,M94}.

Let $\cL$ be a Cameron-Liebler line class with parameter $x$ in $\PG(3,q)$ and let $\cM\subset \cQ^+(5,q)$ be the image of $\cL$ under the Klein correspondence. By Result \ref{res_1}, $\cM$ is a projective two-intersection set in $\PG(5,q)$. By \cite{CK}, we can construct a corresponding strongly regular Cayley graph as follows. First define $D:=\{\lambda v:\,\lambda\in\F_q^*,\; \la v \ra\in \cM\}$, which is a subset of $(\F_q^6,+)$. Then the Cayley graph with vertex set $(\F_q^6,+)$ and connection  set $D$ is strongly regular. Its restricted eigenvalues can be determined as follows. Let $\psi$ be a nonprincipal additive character of $\F_q^6$. Then $\psi$ is principal on a unique hyperplane $P^\perp$ for some $P\in \PG(5,q)$. We have
\begin{align*}
\psi(D)&=\sum_{\la v\ra\in\cM}\sum_{\lambda\in\F_q^*}\psi(\lambda v)
=\sum_{\la v\ra\in\cM}(q[[\la v\ra\in P^\perp]]-1)\\
&=-|\cM|+q|P^\perp\cap\cM|
=\begin{cases}-x+q^3,\; &\textup{if $P\in\mathcal{M}$},\\
-x,\; & \textup{otherwise,}\end{cases}
\end{align*}
where $[[\la v\ra\in P^\perp]]$ is the Kronecker delta function taking value $1$ if $\la v\ra\in P^\perp$ and value $0$ otherwise.
Conversely, for each hyperplane $P^\perp$ of $\PG(5,q)$, we can find a nonprincipal character $\psi$ that is principal on $P^\perp$, and the size of $P^\perp\cap \cM$  can be computed from $\psi(D)$. Therefore, the character values of $D$ reflect the intersection sizes of $\cM$ with the hyperplanes of $\PG(5,q)$. To summarize, we have the following result.

\begin{result}\label{res_charD} Let $\mathcal{L}$ be a set of $x(q^2+q+1)$ lines in $\PG(3,q)$, with $0<x\leq \frac{q^2+1}{2}$, and let $\cM$ be the image of $\cL$ under the Klein correspondence. Define
\[
D:=\{\lambda v:\,\lambda\in\F_q^*,\;\la v \ra\in \cM\}\subset (\F_q^6,+).
\]
Then $\mathcal{L}$ is a  Cameron-Liebler line class with parameter $x$ if and only if $|D|=(q^3-1)x$ and for any $P\in \PG(5,q)$
\begin{align*}
\psi(D)
=\begin{cases}-x+q^3,\quad &\textup{if $P\in\mathcal{M}$},\\
-x,\quad & \textup{otherwise},\end{cases}
\end{align*}
where $\psi$ is any nonprincipal character of $\F_q^6$ that is principal on the hyperplane $P^\perp$.
\end{result}

Following \cite{rodgers} we now introduce a model of the hyperbolic quadric $\cQ^+(5,q)$, which will facilitate our algebraic construction. Let $E=\F_{q^3}$ and $F=\F_q$. We view $E\times E$ as a 6-dimensional vector space over $F$. For a nonzero vector $v\in E\times E$, we use $\la v\ra$ to denote the projective point in $\PG(5,q)$ corresponding to the one-dimensional subspace over $F$ spanned by $v$. Define a quadratic form $Q: E\times E\rightarrow F$ by
\[
Q((x,y))=\tr(xy), \;\forall\, (x,y)\in E\times E,
\]
where $\Tr$ is the relative trace from $E$ to $F$ (that is, for any $x\in E$, $\Tr(x)=x+x^q+x^{q^2}$).  The quadratic form $Q$ is clearly nondegenerate and $Q((x,0))=0$ for all $x\in E$. So $\{\la (x,0)\ra\mid x\in E^*\}$ is a totally isotropic plane with respect to $Q$. It follows that the quadric defined by $Q$ has Witt index 3, and so is hyperbolic. This quadric will be our model for $\cQ^+(5,q)$. Note that for a point $P=\la (x_0,y_0)\ra$, its polar hyperplane $P^\perp$ is given by $P^\perp=\{\la(x,y)\ra:\,\tr(xy_0+x_0y)=0\}$.

Let $\psi_E$ and $\psi_F$ be the canonical additive characters of $E$ and $F$, respectively.
Then each additive character of $E\times E$ has the form 
\begin{equation}\label{fieldchara}
\psi_{a,b}((x,y))=\psi_E(ax+by)=\psi_F(\Tr(ax+by)),\; (x,y)\in E\times E,
\end{equation}
where $(a,b)\in E\times E$. Since $\psi_{a,b}$ is principal on the hyperplane $\{\la(x,y)\ra: \Tr(ax+by)=0\}$,
the character sum condition in Result~\ref{res_charD} can be more explicitly rewritten as
\begin{equation}\label{eqn:polarD}
\psi_{a,b}(D)
=\begin{cases}-x+q^3,\quad &\textup{if $(b,a)\in D$},\\
-x,\quad & \textup{otherwise.}\end{cases}
\end{equation}

\subsection{Preliminaries on characters of finite fields}
In this subsection, we will collect some auxiliary results on Gauss sums. We
assume that the reader is familiar with the  basic theory of characters of finite fields as can found in Chapter~5 of \cite{LN97}.

For a multiplicative character
$\chi$  and the canonical
additive character $\psi$ of $\F_q$, define the {\it Gauss sum} by
\[
G(\chi)=\sum_{x\in \F_q^\ast}\chi(x)\psi(x).
\]
The following are some basic properties of Gauss sums:
\begin{enumerate}
\item[(i)] $G(\chi)\overline{G(\chi)}=q$ if $\chi$ is nonprincipal;
\item[(ii)] $G(\chi^{-1})=\chi(-1)\overline{G(\chi)}$;
\item[(iii)] $G(\chi)=-1$ if $\chi$ is principal.
\end{enumerate}
Let $\gamma$ be a fixed primitive element of $\F_q$ and $k$ a positive integer dividing $q-1$. For $0\leq i\leq k-1$ we set $C_i^{(k,q)}=\gamma^i\langle \gamma^k\rangle$. These are called the {\it $k$th cyclotomic classes} of $\F_q$. The {\it Gauss periods} associated with these cyclotomic classes are defined by $\psi(C_i^{(k,q)}):=\sum_{x\in C_i^{(k,q)}}\psi(x)$, $0\leq i\leq k-1$, where $\psi$ is the canonical additive character of $\F_q$. By orthogonality of characters, the Gauss periods can be expressed as a linear combination of Gauss sums:
\begin{equation}
\psi(C_i^{(k,q)})=\frac{1}{k}\sum_{j=0}^{k-1}G(\chi^{j})\chi^{-j}(\gamma^i), \; 0\leq i\leq k-1,
\end{equation}
where $\chi$ is any fixed multiplicative character of order $k$ of $\F_q$.  For example, if $k=2$, 
we have
\begin{equation}\label{eq:Gaussquad}
\psi(C_i^{(2,q)})=\frac{-1+(-1)^iG(\eta)}{2},\; 0\leq i\leq 1,
\end{equation}
where $\eta$ is the quadratic character of $\F_q$.

The following theorem on Eisenstein sums will be used in the proof of our main theorem in Section 4.
\begin{theorem}\label{thm:Yama}{\em (\cite[Theorem~1]{Yama})}
Let $\chi$ be a nonprincipal multiplicative character of $\F_{q^m}$ and
$\chi'$ be its restriction to $\F_q$. Choose a system $L$ of coset representatives of $\F_q^\ast$ in  $\F_{q^m}^\ast$ in such a way that $L$ can be partitioned into two parts:
\[
L_0=\{x:\,\Tr(x)=0\} \, \, \mbox{and}\, \, L_1=\{x:\,\Tr(x)=1\},
\]
where $\Tr$ is the relative trace from $\F_{q^m}$ to $\F_q$.
Then,
\[
\sum_{x\in L_1}\chi(x)=\left\{
\begin{array}{ll}
G(\chi)/G(\chi'), & \mbox{ if   $\chi'$ is nonprincipal,}\\
-G(\chi)/q, & \mbox{ otherwise}.
 \end{array}
\right.
\]
\end{theorem}
We will also need the  {\it Hasse-Davenport product formula}, which is stated below.
\begin{theorem}
\label{thm:Stickel2}{\em (\cite[Theorem~11.3.5]{BEW97})}
Let $\theta$ be a multiplicative character of order $\ell>1$ of  $\F_{q}$. For  every nonprincipal multiplicative character $\chi$ of $\F_{q}$,
\[
G(\chi)=\frac{G(\chi^\ell)}{\chi^\ell(\ell)}
\prod_{i=1}^{\ell-1}
\frac{G(\theta^i)}{G(\chi\theta^i)}.
\]
\end{theorem}

The Stickelberger theorem on the prime ideal factorization of Gauss sums gives us $p$-adic information on Gauss sums. We will need this theorem to prove a certain divisibility result later on. Let $p$ be a prime, $q=p^f$, and let $\xi_{q-1}$ be a complex primitive $(q-1)$th root of unity. Fix any prime ideal $\mathfrak{P}$ in $\Z[\xi_{q-1}]$ lying over $p$. Then $\Z[\xi_{q-1}]/\mathfrak{P}$ is a finite field of order $q$, which we identify with $\F_q$. Let $\omega_{\mathfrak{P}}$ be the
Teichm\"uller character on $\F_q$, i.e., an isomorphism
$$\omega_{\mathfrak{P}}: \F_q^{*}\rightarrow
\{1,\xi_{q-1},\xi_{q-1}^2,\dots ,\xi_{q-1}^{q-2}\}$$
satisfying
\begin{equation}\label{eq2.3}
\omega_{\mathfrak{P}}(\alpha)\quad ({\rm
mod}\hspace{0.1in}{\mathfrak{P}})=\alpha,
\end{equation}
for all $\alpha$ in $\F_q^*$.
The Teichm\"uller character $\omega_{\mathfrak{P}}$ has order $q-1$. Hence it generates all multiplicative characters of $\F_q$.

Let $\cP$ be the prime ideal of $\Z[\xi_{q-1},\xi_p]$ lying above
$\mathfrak{P}$. For an integer $a$, let
$$s(a)=\nu _{\cP}(G(\omega_{\mathfrak P}^{-a})),$$
where $\nu_{\cP}$ is the $\cP$-adic valuation.
Thus ${\cP}^{s(a)}||G(\omega_{\mathfrak P}^{-a})$.
The following evaluation of
$s(a)$ is due to Stickelberger (see \cite[p.~344]{BEW97}).

 \begin{theorem}\label{stick}
Let $p$ be a prime and $q=p^f$.
For an integer $a$ not divisible by $q-1$,
let $a_0+a_1p+a_2p^2+\cdots +a_{f-1}p^{f-1}$, $0\leq a_i\leq p-1$,
be the $p$-adic expansion of the reduction of $a$ modulo $q-1$.
Then
$$s(a)=a_0+a_1+\cdots +a_{f-1},$$
that is, $s(a)$ is the sum of the $p$-adic digits of the reduction of $a$ modulo $q-1$.
\end{theorem}

\section{The subset $X$ and its properties}\label{sec:setX}

\begin{notation}{\em
Let $q$ be a prime power with $q\equiv 5$ or $9\pmod{12}$ so that $\gcd(q-1,q^2+q+1)=1$. Write $E=\F_{q^3}$, $F=\F_q$, and let $\omega$ be a fixed primitive element of $E$. For any $x\in E^\ast$, we use $\log_\omega(x)$ to denote the integer $0\leq i\leq q^3-2$ such that $x=\omega^i$.   We write $N=q^2+q+1$, and let $w_1$ be an element of order $N$ in $E$ (for example, take $w_1=\omega^{q-1}$). For $x\in F^*$, we define the sign of $x$, $\sgn(x)\in F$, by
\begin{equation}\label{eqn:sgn}
\sgn(x)
=\begin{cases} 1,\quad &\textup{if $x$ is a square},\\
-1,\quad & \textup{otherwise.}\end{cases}
\end{equation}
We also define $\sgn(0):=0$.}
\end{notation}
It is the purpose of this section to introduce a subset $X\subset E$ and prove a few results on $X$, which we will need in the construction of our Cameron-Liebler line classes.

Viewing $E$ as a 3-dimensional vector space over $\F_q$, we will use $E$ as the underlying vector space of $\PG(2,q)$. The points of $\PG(2,q)$ are $\la \omega^i\ra$, $0\leq i\leq q^3-2$, and the lines of $\PG(2,q)$ are
\begin{equation}\label{eqn_Lu}
L_u:=\{\langle x\rangle:\,\Tr(\omega^ux)=0\},
\end{equation}
where $0\leq u\leq q^3-2$. Of course, $\la \omega^i\ra=\la \omega^{i+jN}\ra$ and $L_{u}=L_{u+jN}$, for any $i,j$ and $u$. Note that since $\gcd(q-1, q^2+q+1)=1$, we can also take $\la w_1^i\ra$, $0\leq i\leq q^2+q$, as the points of $\PG(2,q)$.

Define a quadratic form $f: E\rightarrow F$ by $f(x):=\tr(x^2)$, where $\tr$ is the relative trace from $E$ to $F$. The associated bilinear form $B: E\times E\rightarrow F$ is given by $B(x,y)=2\tr(xy)$. It is clear that $B$ is nondegenerate. Therefore $f$ defines a nondegenerate conic $\cQ=\{\langle x\rangle\mid f(x)=0\}$ in $\PG(2,q)$, which contains $q+1$ points. Consequently each line $\l$ of $\PG(2,q)$ meets $\cQ$ in $0$, $1$ or $2$ points, and $\l$ is called a {\it passant}, {\it tangent} or {\it secant line} accordingly. Also it is known that each point $P\in \PG(2,q) \setminus \cQ$ is on either $0$ or $2$ tangent lines to $\cQ$, and $P$ is called an {\it interior} or {\it exterior point} accordingly,  c.f. \cite[p. 158]{stan}.

\begin{lemma}\label{lem_sgn}
With the above notation, we have the following:
\begin{enumerate}
\item The tangent lines to $\cQ$ are given by $L_u$ with $\Tr(\omega^{2u})=0$, $0\leq u\leq q^3-2$.
\item The polarity of $\PG(2,q)$ induced by $\cQ$ interchanges $\la \omega^u\ra$ and $L_u$, where $\Tr(\omega^{2u})=0, 0\leq u\leq q^3-2$, and maps exterior (resp. interior) points to secant (resp. passant) lines.
\item For any point $P=\la v\ra$ off $\cQ$, $P$ is an exterior (resp. interior) point if and only if $f(v)$ has some fixed nonzero sign $\epsilon$ (resp. $-\epsilon$).
\end{enumerate}
\end{lemma}
\begin{proof}(1) and (2) are well-known facts, and we refer the reader to \cite{h1} for proofs. (3) is clear from \cite[p. 166]{stan} in the proof of Theorem 4.3.1 therein. \qed

\end{proof}
\vspace{0.3cm}

Now we define the following subset of $\Z_N$:
\begin{equation}\label{eqn_IQ}
I_\cQ:=\{i:\,0\leq i\leq N-1,\,\tr(w_1^{2i})=0\}=\{d_0,d_1,\ldots, d_{q}\}.
\end{equation}
where the elements are numbered in any (unspecified) order. That is, $\cQ=\{\langle w_1^{d_i}\rangle\mid 0\leq i\leq q\}$.

Let $\beta$ be any element of $F^\ast$ such that $\sgn(\beta)=\epsilon$, with $\epsilon$ as defined in Lemma \ref{lem_sgn}. For $d_0\in I_\cQ$, we define
\begin{equation}\label{eqn_defX}
X:=\{w_1^{d_i}\Tr(w_1^{d_0+d_i}):\,1\leq i\leq q\}\cup\{2 \beta w_1^{d_0}\}
\end{equation}
and
\begin{equation}\label{eqn_defX}
{\overline X}:=\{\log_{\omega}(x)\,(\mod{2N}):\, x\in X\}\subset \Z_{2N}.
\end{equation}

\begin{lemma}\label{lem_sgn_prod} Let $d_i,d_j,d_k$ be three distinct elements of $I_\cQ$.  Then the sign of $$B(w_1^{d_i},w_1^{d_j})B(w_1^{d_i},w_1^{d_k})B(w_1^{d_j},w_1^{d_k})$$ is equal to the sign of $f(v)$ for any exterior point $\la v\ra$. In other words, $2\Tr(w_1^{d_i+d_j})\Tr(w_1^{d_i+d_k})\Tr(w_1^{d_j+d_k})$ has sign $\epsilon$, where $\epsilon$ is the same as in part (3) of Lemma \ref{lem_sgn}. In particular, $\epsilon=1$.
\end{lemma}
\begin{proof} Since $\cQ$ is a conic, $w_1^{d_i},w_1^{d_j},w_1^{d_k}$ are linearly independent over $F$ and thus form a basis of $E$ over $F$.  The Gram matrix of the bilinear form $B$ with respect to this basis is equal to
\begin{align}\label{eqn_mat}
\begin{pmatrix}
B(w_1^{d_i},w_1^{d_i})&B(w_1^{d_i},w_1^{d_j})&B(w_1^{d_i},w_1^{d_k})\\
B(w_1^{d_j},w_1^{d_i})&B(w_1^{d_j},w_1^{d_j})&B(w_1^{d_j},w_1^{d_k})\\
B(w_1^{d_k},w_1^{d_i})&B(w_1^{d_k},w_1^{d_j})&B(w_1^{d_k},w_1^{d_k})
\end{pmatrix}
\end{align}
which is symmetric with diagonal entries equal to $0$. Its determinant is equal to
\[2B(w_1^{d_i},w_1^{d_j})B(w_1^{d_i},w_1^{d_k})B(w_1^{d_j},w_1^{d_k}).\]

Let $P=\la v\ra$ be an exterior point, say, $P$ is the intersection of the tangent lines through $\la w_1^{d_i}\ra$ and $\la w_1^{d_j}\ra$. Then $B(v,w_1^{d_i})=B(v,w_1^{d_j})=0$, and the Gram matrix with respect to the basis $v, w_1^{d_i},w_1^{d_j}$  has determinant $-2f(v)B(w_1^{d_i},w_1^{d_j})^2$. Since $q\equiv 1\pmod{4}$, $-1$ is a square in $\F_q^*$. By \cite[p. 262]{stan}, the two determinants have the same sign. This proves the first part of the lemma. It remains to prove that $\epsilon=1$. We observe that the matrix \eqref{eqn_mat} can be written as $2MM^{\top}$ with
\[
M=\begin{pmatrix}
w_1^{d_i}&w_1^{qd_i}&w_1^{q^2d_i}\\
w_1^{d_j}&w_1^{qd_j}&w_1^{q^2d_j}\\
w_1^{d_k}&w_1^{qd_k}&w_1^{q^2d_k}
\end{pmatrix}.
\]
We claim that the determinant of $M$, $\det(M)$, is in $F$. To see this, applying the Frobenius automorphism $\sigma:\,x\mapsto x^{q}$ of $Gal(E/F)$ to $M$ entry-wise, we get
\[
\sigma(\det(M))=\det(\sigma(M))=\begin{vmatrix}
w_1^{qd_i}&w_1^{q^2d_i}&w_1^{d_i}\\
w_1^{qd_j}&w_1^{q^2d_j}&w_1^{d_j}\\
w_1^{qd_k}&w_1^{q^2d_k}&w_1^{d_k}
\end{vmatrix}=\det(M).
\]
The claim that $\epsilon=1$ now follows from this fact and the first part of the lemma.
\qed
\end{proof}

\begin{lemma}\label{lem_def_X}
With the above notation, if we use any other $d_i$ in place of $d_0$ in the definition of $X$, then the resulting set $X'$ satisfies that
$\overline{X'}\equiv \overline{X}\,(\mod{2N})$ or $\overline{X'}\equiv \overline{X}+N\,(\mod{2N})$.
\end{lemma}
\begin{proof} Without loss of generality, we assume that we use $d_1$ in place of $d_0$ in the definition of $X$, and obtain  $X'=\{w_1^{d_i}\Tr(w_1^{d_1+d_i}):\,0\leq i\leq q, \;i\ne 1\}\cup\{2 \beta  w_1^{d_1}\}$. We have the following observations.
\begin{enumerate}
\item For $i\ne 0,1$, that is, $2\leq i\leq q$, the sign of the quotient of $w_1^{d_i}\Tr(w_1^{d_0+d_i})$ and  $w_1^{d_i}\Tr(w_1^{d_1+d_i})$ is a constant: their quotient $\Tr(w_1^{d_0+d_i})\Tr(w_1^{d_1+d_i})^{-1}$ lies in $F$, and its sign is clearly equal to that of $\Tr(w_1^{d_0+d_i})\Tr(w_1^{d_1+d_i})$; this sign is equal to  $\sgn(2\tr(w_1^{d_0+d_1}))$ by Lemma \ref{lem_sgn_prod}.

\item The sign of the quotient of $2 \beta  w_1^{d_1} $ and  $w_1^{d_1}\Tr(w_1^{d_0+d_1})$ is equal to $\sgn(2\beta\Tr(w_1^{d_1+d_0}))=\sgn(2\Tr(w_1^{d_1+d_0}))$, since we have chosen $\beta$ such that $\sgn(\beta)=\epsilon=1$. Similarly the sign of the quotient of $2 \beta  w_1^{d_0} $ and $w_1^{d_0}\Tr(w_1^{d_1+d_0})$ is  equal to $\sgn(2\Tr(w_1^{d_1+d_0}))$.
\end{enumerate}
The above observations imply that there is a pairing of the elements of $X'$ and $X$, say, $p: X'\rightarrow X$ a bijection, such that $x'/p(x')$ are nonzero squares in $F$ for all $x'\in X'$ or $x'/p(x')$ are nonsquares in $F$ for all $x'\in X'$. Upon taking logarithm and modulo $2N$ we get the conclusion of the lemma.
\qed
\end{proof}
\vspace{0.3cm}
\begin{remark}\label{rem_aut_pre}It is clear that $d_0':=qd_0\pmod{N}$ is an element of $I_\cQ$. In Lemma \ref{lem_def_X}, consider the special case where we replace $d_0$ by $d_0'$ in the definition of $X$, and denote the resulting set by $X'$. We observe that
\begin{align*}
2\Tr(w_1^{d_0'+d_0})&=\Tr(w_1^{d_0})^2-\Tr(w_1^{2d_0})-\Tr(w_1^{2d_0})^q=(\Tr(w_1^{d_0})^2.
\end{align*} 
Consequently $\sgn(2\Tr(w_1^{d_0'+d_0}))=1$. It follows that $\overline{X'}\equiv \overline{X}\,(\mod{2N})$ in this particular case. By definition, we have $\overline{X'}\equiv q\overline{X}\pmod{N}$, so we have shown that the subset
$\overline{X}$ is invariant under multiplication by $q$. This fact will be needed in the next section when we discuss automorphism groups of the newly construected Cameron-Liebler line classes.
\end{remark}

We now prove some properties of the set $X$ which will be needed in the next section.  Let $S$ (resp. $N$) be the set of nonzero squares (resp. nonsquares) of $F$, and write $$s=\sum_{x\in S}\psi_F(x), \;n=\sum_{x\in N}\psi_F(x),$$
 where $\psi_F$ is the canonical additive character of $F$. For $0\leq u\leq q^3-2$, we define the exponential sums
\begin{equation}
T_u:=\sum_{i\in \overline{X}}\sum_{x\in S}\psi_F(\Tr(\omega^{u+i})x)\notag
\end{equation}
 To simplify notation, we often write $T_u= \sum_{i\in \overline{X}}\psi_F(\Tr(\omega^{u+i})S).$ Note that $i\in \overline{X}$ if and only if $i\equiv \log_{\omega}(x)\pmod {2N}$ for some $x\in X$, which in turn is equivalent to $x=\omega^{i+2Nj}$ for some integer $j$. Since $\omega^{2Nj}$ is an element of $S$, we can view $\omega^i$ in the above definition of $T_u$ as coming from $X$. It follows that
\begin{equation}\label{eqn_Tu}
T_u=\psi_F(\tr(2\beta \omega^uw_1^{d_0})S)+\sum_{i=1}^q\psi_F(\tr(\omega^u w_1^{d_i})\tr(w_1^{d_i+d_0})S)
\end{equation}
We will evaluate these sums explicitly.

\begin{remark}\label{rem_Tu} {\em The following are some simple observations.
\begin{enumerate}
\item It is clear that each summand in the right hand side of (\ref{eqn_Tu}) is equal to one of $|S|=\frac{q-1}{2}$, $s$ and $n$, depending on the sign of the trace term in front of $S$.
\item If we replace $X$ by $X'$ as in Lemma~\ref{lem_def_X}, then the value of $T_u$ is either unchanged or is equal to $T_{u+N}$.
\end{enumerate}}
\end{remark}

\begin{theorem} \label{thm:main2}Let $T=\{\log_{\omega}(x)\,(\mod{N})\,|\,
\Tr(x)=0\}$ and $T'=\{\log_\omega(x)\,(\mod{2N})\,|\, \Tr(x)=1\}$.
The exponential sums $T_u$ take the following four values:
\[
T_u=\left\{
\begin{array}{ll}
\frac{q-1}{2}+qs \mbox{ or } \frac{q-1}{2}+qn, & \mbox{ if $2u \;(\mod{N}) \in T$,}\\
-\frac{q+1}{2}, & \mbox{ if $2u \;(\mod{N})\not \in T$ and  $2u \; (\mod{2N})\not \in T'$}, \\
\frac{q-1}{2}, & \mbox{ if $2u \;(\mod{N})\not \in T$ and  $2u \; (\mod{2N}) \in T'$,}
 \end{array}
\right.
\]
where $s$ and $n$ are defined as above, that is, $s=\sum_{x\in S}\psi_F(x)$ and  $n=\sum_{x\in N}\psi_F(x)$.
\end{theorem}

\begin{proof}
 We consider the following three cases according to the line $L_u$ as defined in \eqref{eqn_Lu} is a tangent, passant, or secant.

{\bf Case 1: $L_u$ is a tangent line.} In this case, $\omega^u$ is a zero of $f$ by (1) of Lemma \ref{lem_sgn}, where we recall that $f(x):=\Tr(x^2)$, so $\la \omega^u\ra=\la w_1^{d_k}\ra$ for some $0\leq k\leq q$ by the definition of $I_\cQ$ in Eqn. \eqref{eqn_IQ}.
Note that $u$ satisfies $2u \;(\mod{N}) \in T$ by the definition of $f$ and $T$. In view of (2) of Remark \ref{rem_Tu}, we may assume that $d_k=d_0$.
(If necessary, replace $u$ by $u+N$, and then the resulting $\omega^u$ is still
a zero of $f$, i.e., satisfies $2u\; (\mod{N}) \in T$.) Now $\tr(\omega^u w_1^{d_i})=0$ if and only if $\la w_1^{d_i}\ra$ lies on the tangent line $L_u$, i.e., $d_i=d_0$. We see that the elements  $\tr(w_1^{d_0+d_i})^2$, $i\ne 0$, are all nonzero squares, so that $T_u=\frac{q-1}{2}+qs$ or $\frac{q-1}{2}+qn$ by (1) of Remark \ref{rem_Tu}.
Note that if we replace $u$ by $u+N$, then the value of $T_u$ is replaced by the other in this case.  Hence, $\{T_u,T_{u+N}\}=\{\frac{q-1}{2}+qs,\frac{q-1}{2}+qn\}$.

{\bf Case 2: $L_u$ is a passant line.} In this case, $\la \omega^u\ra$ is an interior point and thus $f(\omega^u)$ has sign $-\epsilon(=-1)$ by Lemmas~\ref{lem_sgn}
and \ref{lem_sgn_prod}.
Note that $u$ satisfies $2u \;(\mod{N})\not \in T$ and  $2u\; (\mod{2N})\not \in T'$ since $f(\omega^u)$ is a nonsquare of $F$. Each line through $\la \omega^u\ra$ has either $0$ or $2$ points of $\cQ$, so the points of $\cQ$ are partitioned into pairs accordingly.  Let $\la w_1^{d_i}\ra,\la w_1^{d_j}\ra$ be two points  of $\cQ$ that lie on a secant line through $\la \omega^u\ra$.

Since the three points $\la w_1^{d_i}\ra,\la w_1^{d_j}\ra$ and $\la \omega^u\ra$ are collinear, the Gram matrix of $w_1^{d_i},w_1^{d_j},\omega^u$ is singular. By direct computations we find that the determinant of this Gram matrix is equal to $B(w_1^{d_i},w_1^{d_j})^2B(\omega^u,\omega^u)+2B(w_1^{d_i},\omega^u)B(w_1^{d_j},\omega^u)B(w_1^{d_i},w_1^{d_j}).$ It follows that
\begin{equation}\label{deteq}
B(w_1^{d_i},w_1^{d_j})^2B(\omega^u,\omega^u)+2B(w_1^{d_i},\omega^u)B(w_1^{d_j},\omega^u)B(w_1^{d_i},w_1^{d_j})=0
\end{equation}
Recall that $f(\omega^u)=\tr(\omega^{2u})$ has sign $-\epsilon=-1$, and $-1$ is a square since $q\equiv 1\pmod{4}$. We see from (\ref{deteq}) that $ 2\Tr(w_1^{d_i}\omega^u)\Tr(w_1^{d_j}\omega^u) \Tr(w_1^{d_i+d_j})$ has sign $-\epsilon=-1$. If either $d_i$ or $d_j$ is $d_0$, say $d_i$ is $d_0$, then $\Tr(2\beta \omega^uw_1^{d_0})\Tr(\omega^uw_1^{d_j})\Tr(w_1^{d_j+d_0})$ is a nonsquare since $\beta\in F^\ast$ has sign $\epsilon=1$. Hence exactly one of  $\Tr(2\beta \omega^uw_1^{d_0})$ and $\Tr(\omega^uw_1^{d_j})\Tr(w_1^{d_j+d_0})$ is a square and the other is a  nonsquare of $F$. Their corresponding summands in the right hand side of (\ref{eqn_Tu})  thus contribute $s+n=-1$ to $T_u$. On the other hand, if $d_i,d_j\not=d_0$, by Lemma~\ref{lem_sgn_prod}, $2\Tr(w_1^{d_i+d_j})\Tr(w_1^{d_i+d_0})\Tr(w_1^{d_j+d_0})$ has sign $\epsilon=1$, so $ \Tr(w_1^{d_i}\omega^u)\Tr(w_1^{d_j}\omega^u) \Tr(w_1^{d_i+d_0})\Tr(w_1^{d_j+d_0})$ has sign $-1$, i.e., it is a nonsquare of $F$. Again exactly one of $ \Tr(w_1^{d_i}\omega^u)\Tr(w_1^{d_i+d_0})$ and $\Tr(w_1^{d_j}\omega^u) \Tr(w_1^{d_j+d_0})$ is a square and the other is a nonsquare of $F$. Their corresponding summands in the right hand side of (\ref{eqn_Tu})  thus contribute $s+n=-1$ to $T_u$.

Summing up we have $|T_u|=-\frac{q+1}{2}$ in this case.
Note that replacing $u$ by $u+N$ does not  change the value of $T_u$  in this case.
Hence, $T_u=T_{u+N}=-\frac{q+1}{2}$.

{\bf Case 3: $L_u$ is a secant line.}  In this case, $\la \omega^u\ra$ is an exterior point, and thus $f(\omega^u)$ has sign $\epsilon(=1)$ by Lemma~\ref{lem_sgn}
and \ref{lem_sgn_prod}.
Note that $u$ satisfies $2u \;(\mod{N})\not \in T$ and  $2u \;(\mod{2N}) \in T'$ since $f(\omega^u)$ is a square. By (2) of Remark \ref{rem_Tu}, we may assume that it is on the tangent line $l$ of $\la w_1^{d_0}\ra$.
(If necessary, replace $u$ by $u+N$, and then the resulting $u$  still satisfies that $2u \;(\mod{N})\not \in T$ and  $2u \;(\mod{2N}) \in T'$.)
For $i\ne 0$, the tangent line through $\la w_1^{d_i}\ra$ intersects $l$ at the point $\la \omega^u-\lambda_i w_1^{d_0}\ra$ by straightforward calculations, where $\lambda_i=\Tr(\omega^uw_1^{d_i})\Tr(w_1^{d_i+d_0})^{-1}$.  When $i$ ranges from $1$ to $q$, we get all the points of $l$ other than $\la w_1^{d_0}\ra$, since each other tangent line intersects $l$ at a distinct point. This implies that  the values $\lambda_i$ are all distinct, and hence $\{\lambda_i:\,1\leq i\leq q\}$ is equal to $F$. Since $\Tr(\omega^uw_1^{d_i})\Tr(w_1^{d_i+d_0})$ has the same sign as $\lambda_i$, we see that the $q$ terms in the second sum in the right hand side of (\ref{eqn_Tu}) are $\frac {q-1}{2}$ once, $s$ with multiplicity $\frac {q-1}{2}$, and $n$ with multiplicity $\frac {q-1}{2}$. Also the first summand in the right hand side of (\ref{eqn_Tu}) is equal to $\frac{q-1}{2}$ since $\la \omega^u\ra$ is on the tangent line $l$ of $\la w_1^{d_0}\ra$.  Summing up, in this case, we have $|T_u|=2\cdot\frac{q-1}{2}+\frac{q-1}{2}(s+n)=\frac{q-1}{2}$. Note that the replacing of $u$ by $u+N$ does not change the value of $T_u$  in this case.  Hence, $T_u=T_{u+N}=\frac{q-1}{2}$.

This completes the proof of the theorem.\qed
\end{proof}
\begin{remark}\label{tuvalue}{\em
\begin{itemize}
\item[(i)]
Some parts of the above theorem can be made more precise. We have $T_u=\frac{q-1}{2}+qs$ if and only if
\[
u\,(\mod{2N}) \in
\left\{
\begin{array}{ll}
\overline{X}, & \mbox{ when  $\sgn(2)=1$,}\\
\overline{X}+N, & \mbox{ when $\sgn(2)=-1$}.
 \end{array}
\right.
\]
and
$T_u=\frac{q-1}{2}+qn$ if and only if
\[
u\,(\mod{2N}) \in
\left\{
\begin{array}{ll}
\overline{X}, & \mbox{ when $\sgn(2)=-1$,}\\
\overline{X}+N, & \mbox{ when $\sgn(2)=1$}.
 \end{array}
\right.
\]

To see this, by Eqn.~(\ref{eqn_Tu}), consider the set
\begin{align*}
&\, \{\sgn(\Tr(\omega^u\cdot c)):c\in X\}\\
=&\, \{\sgn(\Tr(\omega^uw_1^{d_i})\Tr(w_1^{d_0+d_i})):1\le i\le q\}
\cup \{\sgn(2\beta\Tr(\omega^uw_1^{d_0}))\}.
\end{align*}
For $\omega^u\in X$, by Lemma~\ref{lem_sgn_prod}, we have the
following observations:
\\
(1) If $\omega^u=w_1^{d_j}\Tr(w_1^{d_j+d_0})$, then
\[
\sgn(\Tr(\omega^uw_1^{d_i})\Tr(w_1^{d_0+d_i}))=
\sgn(\Tr(w_1^{d_j+d_i})\Tr(w_1^{d_i+d_0})\Tr(w_1^{d_j+d_0}))=\sgn(2)
\]
and
\[
\sgn(2\beta\Tr(\omega^uw_1^{d_0}))=\sgn(2\beta)\sgn(\Tr(w_1^{2d_0}))=0.
\]
(2) If $\omega^u=2\beta \Tr(w_1^{d_0})$, then
\[
\sgn(\Tr(\omega^uw_1^{d_i})\Tr(w_1^{d_0+d_i}))=
\sgn(2\beta\Tr(w_1^{d_i+d_0})^2)=\sgn(2)
\]
and
\[
\sgn(2\beta\Tr(\omega^uw_1^{d_0}))=\sgn(2\beta)\sgn(\Tr(w_1^{2d_0}))=0.
\]
Summing up, 
we have $T_u=\frac{q-1}{2}+qs$ (resp. $\frac{q-1}{2}+qn$) when $\sgn(2)=1$ (resp. $\sgn(2)=-1$).
Similarly, for $\omega^u\in \omega^N X$, it holds that
$T_u=\frac{q-1}{2}+qs$ (resp. $\frac{q-1}{2}+qn$) when $\sgn(2)=-1$ (resp. $\sgn(2)=1$). By Theorem~\ref{thm:main2}, the converse is also true.
\item[(ii)]
By (\ref{eq:Gaussquad}), the values of $s$ and $n$ are given by $s=\frac{-1+G(\eta')}{2}$ and $n=\frac{-1-G(\eta')}{2}$,
where $\eta'$ is the quadratic character of $F$.
Thus,
the condition of the remark above can be described as follows: $T_u=\frac{q-1}{2}+q\frac{-1+\eta'(2)G(\eta')}{2}$ (resp.
$T_u=\frac{q-1}{2}+q\frac{-1-\eta'(2)G(\eta')}{2}$) if and only if
$u\,(\mod{2N})\in \overline{X}$ (resp. $u\,(\mod{2N})\in \overline{X}+N$).
\end{itemize} }
\end{remark}
\section{A construction of Cameron-Liebler line classes
with parameter $x=\frac{q^2-1}{2}$}\label{proofPDS}
In this section, we give the promised construction of Cameron-Liebler line classes with parameter $x=\frac{q^2-1}{2}$. We will use the same notation introduced in previous sections. By Result~\ref{res_1}, it suffices to construct an $x$-tight set $\cM\in\cQ^+(5,q)$ satisfying the hyperplane intersection property specified in Result~\ref{res_1}. Going from a subset of points in $\PG(5,q)$ to a subset of vectors in $E\times E$, by Result~\ref{res_charD}, it suffices to construct a subset $D\subset E\times E \setminus\{(0,0)\}$ such that $|D|=x(q^3-1)$, $Q((u,v))=0$ for all $(u,v)\in D$, and the additive character values of $D$ satisfy \eqref{eqn:polarD}.

The first step of our construction follows the idea in \cite[p.~37]{Rthesis}. That is, we prescribe an automorphism group for the $x$-tight set that we intend to construct. We will take the model of $\cQ^{+}(5,q)$ as introduced in Section~\ref{sec:pre}.  Define the map $g$ on $\cQ^+(5,q)$ by
$$g: (x,y)\mapsto (w_1 x, w_1^{-1} y),$$
where $w_1\in E^*$ has order $N=q^2+q+1$. Then the cyclic subgroup $C\leq PGO^+(6,q)$ generated by $g$ has order $q^2+q+1$, and it acts semi-regularly on the points of $\cQ^+(5,q)$ (so each orbit contains $q^2+q+1$ points).  The $x$-tight set $\cM$ will be a union of orbits of $C$ acting on $\cQ^+(5,q)$. The main difficulty with this approach lies in coming up with a general choice of orbits of $C$ for all $q$ so that the union of the chosen orbits is an $x$-tight set. We will use the subset $X$ introduced in Section 3 to help choose orbits for our purpose.

Let ${\overline X}$ be the set as defined in (\ref{eqn_defX}).  The set ${\overline X}$ can be expressed as
\[
\overline{X}=2A\cup (2B+N)\,(\mod{2N})
\]
for some  $A,B \subseteq \Z_N$ with $|A|+|B|=q+1$. Define a subset $I_X$ of $\Z_{4N}$ by
\begin{equation}\label{eqn_defIX}
I_X:=\{4t+Ns \,(\mod{4N}): t \in A,s=0,1\}
\cup \{4t+Ns \,(\mod{4N}): t \in B,s=2,3\}.
\end{equation}
That is, $$I_X=4A \cup (4A+N)\cup (4B+2N)\cup (4B+3N).$$ Note that $|I_X|=2(q+1).$
Now we use $I_X$ to give the main construction of this paper. Let $q$ be a prime power such that $q\equiv 5$ or $9\pmod{12}$. Define
\begin{equation}\label{def:D}
D:=\{(xy,xy^{-1}z\omega^{\ell})\,|\,x\in F^\ast,y\in C_0^{(q-1,q^3)},z\in C_0^{(4N,q^3)},\ell \in I_X\}\subseteq
E\times E,
\end{equation}
where $C_0^{(q-1,q^3)}:=\langle \omega^{q-1}\rangle$, $C_0^{(4N,q^3)}:=\langle \omega^{4N}\rangle$. Clearly we have $$|D|=(q-1)(q^2+q+1)\frac{q-1}{4}\cdot 2(q+1)=\frac{q^2-1}{2}(q^3-1),$$
and $\lambda D=D$ for all $\lambda\in F^*$. Also it is clear that $D$ is a subset of the hyperbolic quadric $\{(x,y)\in E^2\,|\,\Tr(xy)=0\}$ since $\Tr(\omega^{\ell})=0$ for any $\ell\in I_X$ by the definition of $I_X$ and $X$. Let $\cM$ be the set of projective points in $\PG(5,q)$ corresponding to $D$. Then $|\cM|=\frac{q^2-1}{2}(q^2+q+1)$ and $\cM\subset \cQ^+(5,q)$.

\begin{theorem}\label{thm:main} The line set ${\mathcal L}$ in $\PG(3,q)$ corresponding to $\cM$ under the Klein correspondence forms a Cameron-Liebler line class with parameter $x=\frac{q^2-1}{2}$.
\end{theorem}

In the rest of this section we give a proof of Theorem~\ref{thm:main}. By the discussions at the beginning of this section, it suffices to show that for all $(0,0)\neq (a,b)\in E\times E$,
\begin{align*}
\psi_{a,b}(D)
=\begin{cases}-\frac{q^2-1}{2}+q^3,\quad &\textup{if $(b,a)\in D$},\\
-\frac{q^2-1}{2},\quad & \textup{otherwise}.\end{cases}
\end{align*}
By (\ref{fieldchara}) and the definition of $D$, $\psi_{a,b}(D)$ is
expressed as
\begin{equation}\label{eigen1}
S_{a,b}:=\sum_{\ell\in I_X}\sum_{i=0}^{N-1}\sum_{j=0}^{\frac{q-1}{4}-1}\sum_{k=0}^{q-2}\psi_E(a\omega^{Nk}\omega^{(q-1)i}+b\omega^{Nk}
\omega^{-(q-1)i}\omega^{4Nj+\ell}).
\end{equation}
We evaluate these character sums by considering two cases: (i) $a=0$ or
$b=0$; and (ii) $a\not=0$ and $b\not=0$.
\begin{lemma}\label{lemma:ab0}
Assume either $a=0$ or $b=0$. Then the quadruple sum $S_{a,b}$ in $(\ref{eigen1})$ is equal to $-\frac{q^2-1}{2}$.
\end{lemma}
\proof
If  $a\not=0$ and  $b=0$,  since $\{\omega^{Nk+(q-1)i}\,|\,0\le k\le q-2,0\le i\le N-1\}=E^\ast$, we have
\begin{align*}
S_{a,b}= &\,\sum_{\ell\in I_X}\sum_{i=0}^{N-1}\sum_{j=0}^{\frac{q-1}{4}-1}\sum_{k=0}^{q-2}\psi_E(a\omega^{Nk}\omega^{(q-1)i})\\
=&\,\frac{q-1}{4}|I_X|\sum_{x\in E^\ast}\psi_E(ax)=-\frac{1}{4}(q-1)|I_X|=-\frac{q^2-1}{2}.
\end{align*}
Similarly, if $a=0$ and $b\not=0$, we have
\begin{align*}
S_{a,b}=&\,\sum_{\ell\in I_X}\sum_{i=0}^{N-1}\sum_{j=0}^{\frac{q-1}{4}-1}\sum_{k=0}^{q-2}\psi_E(b\omega^{Nk}\omega^{-(q-1)i}\omega^{4Nj+\ell})\\
=&\,\sum_{\ell\in I_X}\sum_{j=0}^{\frac{q-1}{4}-1}\sum_{x\in E^\ast}\psi_E(bx\omega^{4Nj+\ell})=-\frac{1}{4}(q-1)|I_X|=-\frac{q^2-1}{2}.
\end{align*}
This completes the proof of the lemma.
\qed
\vspace{0.3cm}

Next, we consider Case (ii): $a\neq 0$ and $b\not=0$. The quadruple sum $S_{a,b}$ in this case can be essentially reduced to the character sum  $T_u$, which was already evaluated in
Theorem~\ref{thm:main2} (and Remark~\ref{tuvalue}). The computations involved in reducing $S_{a,b}$ to $T_u$ are routine but complicated.
So we start with a lemma, which says that the sum $S_{a,b}$ can be expressed in terms of Gauss sums of order $4N$ of $E$.
\begin{lemma}\label{lem:abnot0}
Assume that $a,b\not=0$. Then
\begin{equation}
S_{a,b}=
\frac{(q-1)}{4(q^3-1)}\sum_{\ell\in I_X}\sum_{\substack{u,v=0\\u+v\equiv 0(\mod{4})\\u\equiv v(\mod{N})}}^{4N-1}G(\chi_{4N}^{-v})G(\chi_{4N}^{-u})
\chi_{4N}^{v}(a)\chi_{4N}^u(b)
\chi_{4N}^u(\omega^\ell),\label{eigen34}
\end{equation}
where $\chi_{4N}$ is a multiplicative character of order $4N$ of $E$.
\end{lemma}
\proof
By orthogonality of characters,  the quadruple sum $S_{a,b}$ is equal to
{\footnotesize
\begin{align}
&\,\frac{1}{4}\sum_{\ell\in I_X}\sum_{i=0}^{N-1}\sum_{x,y\in F^\ast}\psi_E(ax\omega^{(q-1)i}+bx\omega^{-(q-1)i}y^4\omega^\ell)\nonumber\\
=&\,\frac{1}{4(q^3-1)^2}\sum_{\ell\in I_X}\sum_{i=0}^{N-1}\sum_{x,y\in F^\ast}\sum_{j,k=0}^{q^3-2}
G(\chi_{q^3-1}^{-j})\chi_{q^3-1}^j(ax\omega^{(q-1)i})
G(\chi_{q^3-1}^{-k})\chi_{q^3-1}^k(bx\omega^{-(q-1)i}y^4\omega^\ell)\nonumber\\
=&\,\frac{1}{4(q^3-1)^2}\sum_{\ell\in I_X}\sum_{i=0}^{N-1}\sum_{x\in F^\ast}\sum_{j,k=0}^{q^3-2}G(\chi_{q^3-1}^{-j})G(\chi_{q^3-1}^{-k})
\chi_{q^3-1}^j(a)\chi_{q^3-1}^k(b)
\chi_{q^3-1}^{j+k}(x)\chi_{q^3-1}^{j-k}(\omega^{(q-1)i})
\chi_{q^3-1}^k(\omega^\ell)\nonumber\\
&\,\hspace{11cm}\times \Big(\sum_{y\in F^\ast}\chi_{q^3-1}^k(y^4)\Big).
\label{eigen2}
\end{align}
}
Let $\chi_{4N}=\chi_{q^3-1}^\frac{q-1}{4}$. Since
$\sum_{y\in F^\ast}\chi_{q^3-1}^k(y^4)=q-1$ or $0$ according to
$\frac{(q-1)}{4}\,|\,k$ or not, continuing from (\ref{eigen2}), we have
{\footnotesize
\begin{align}
S_{a,b}=&\,\frac{q-1}{4(q^3-1)^2}\sum_{\ell\in I_X}\sum_{i=0}^{N-1}\sum_{x\in F^\ast}\sum_{j=0}^{q^3-2}\sum_{u=0}^{4N-1}G(\chi_{q^3-1}^{-j})G(\chi_{4N}^{-u})
\chi_{q^3-1}^j(a)\chi_{4N}^u(b)
\chi_{q^3-1}^{j+\frac{q-1}{4}u}(x)\chi_{q^3-1}^{j-\frac{q-1}{4}u}(\omega^{(q-1)i})
\chi_{4N}^u(\omega^\ell)\nonumber\\
=&\,\frac{q-1}{4(q^3-1)^2}\sum_{\ell\in I_X}\sum_{i=0}^{N-1}\sum_{j=0}^{q^3-2}\sum_{u=0}^{4N-1}G(\chi_{q^3-1}^{-j})G(\chi_{4N}^{-u})
\chi_{q^3-1}^j(a)\chi_{4N}^u(b)
\chi_{q^3-1}^{j-\frac{q-1}{4}u}(\omega^{(q-1)i})
\chi_{4N}^u(\omega^\ell)\nonumber\\
&\,\hspace{11cm}\times\Big(\sum_{x\in F^\ast}\chi_{q^3-1}^{j+\frac{q-1}{4}u}(x)\Big). \label{eigen3}
\end{align}
}
Let $\chi_{N}=\chi_{q^3-1}^{q-1}$. Since
$\sum_{x\in F^\ast}\chi_{q^3-1}^{j+\frac{q-1}{4}u}(x)=q-1$ or $0$ according to
$j\equiv 0\,(\mod{\frac{q-1}{4}})$ and $j+\frac{q-1}{4}u\equiv 0\,(\mod{q-1})$ or not, continuing from (\ref{eigen3}), we have
\begin{align}
S_{a,b}=&\,\frac{(q-1)^2}{4(q^3-1)^2}\sum_{\ell\in I_X}\sum_{i=0}^{N-1}\sum_{\substack{u,v=0\\ u+v\equiv 0(\mod{4})}}
^{4N-1}
G(\chi_{4N}^{-v})G(\chi_{4N}^{-u})
\chi_{4N}^{v}(a)\chi_{4N}^u(b)
\chi_{4N}^{v-u}(\omega^{(q-1)i})
\chi_{4N}^u(\omega^\ell)\nonumber\\
=&\,\frac{(q-1)^2}{4(q^3-1)^2}\sum_{\ell\in I_X}\sum_{\substack{u,v=0\\u+v\equiv 0(\mod{4})}}^{4N-1}G(\chi_{4N}^{-v})G(\chi_{4N}^{-u})
\chi_{4N}^{v}(a)\chi_{4N}^u(b)
\chi_{4N}^u(\omega^\ell)\Big(\sum_{i=0}^{N-1}\chi_{N}^{(v-u)\frac{q-1}{4}}(\omega^i)\Big)\nonumber\\
=&\,\frac{(q-1)}{4(q^3-1)}\sum_{\ell\in I_X}\sum_{\substack{u,v=0\\u+v\equiv 0(\mod{4})\\u\equiv v(\mod{N})}}^{4N-1}G(\chi_{4N}^{-v})G(\chi_{4N}^{-u})
\chi_{4N}^{v}(a)\chi_{4N}^u(b)
\chi_{4N}^u(\omega^\ell).
\end{align}
The proof of the lemma is now complete. \qed
\vspace{0.3cm}

\begin{remark}\label{re:transeq}{\em
In (\ref{eigen34}) of Lemma~\ref{lem:abnot0},  $v$ is expressed as
$v=-Nc+4d$
if $u=Nc+4d$ for $c=0,1,2,3,$ and $d=0,1,\ldots,N-1$.
Write $\chi_4=\chi_{4N}^N$ and we can assume that $\chi_4(\omega^N)=i$, where
$i=\sqrt{-1}$.
Then, by the definition of $I_X$,  the sum in (\ref{eigen34}) can be expanded as follows:
\begin{align}
\frac{4(q^3-1)}{q-1}S_{a,b}
=&\,\sum_{s=0,1}\sum_{t\in A}\sum_{c=0,1,2,3}\sum_{d=0}^{N-1}G(\chi_{4N}^{Nc-4d})G(\chi_{4N}^{-Nc-4d})
\chi_4^c(a^{-1}b)\chi_N^d(ab)
\chi_{4N}^{Nc+4d}(\omega^{4t+Ns})\nonumber\\
&\,\hspace{0.3cm}+\sum_{s=2,3}\sum_{t\in B}\sum_{c=0,1,2,3}\sum_{d=0}^{N-1}G(\chi_{4N}^{Nc-4d})
G(\chi_{4N}^{-Nc-4d})
\chi_4^c(a^{-1}b)\chi_N^d(ab)
\chi_{4N}^{Nc+4d}(\omega^{4t+Ns})\nonumber\\
=&\,\sum_{t\in A}\sum_{c=0,1,2,3}\sum_{d=0}^{N-1}G(\chi_{4N}^{Nc-4d})
G(\chi_{4N}^{-Nc-4d})
\chi_4^c(r')\chi_N^d(r)
\chi_{N}^{d}(\omega^{4t})(1+(i)^c)\nonumber\\
&\,\hspace{0.3cm} +\sum_{t\in B}\sum_{c=0,1,2,3}
\sum_{d=0}^{N-1}G(\chi_{4N}^{Nc-4d})G(\chi_{4N}^{-Nc-4d})
\chi_4^c(r')\chi_N^d(r)
\chi_{N}^{d}(\omega^{4t})((-1)^c+(-i)^c)
 \label{eigen4}
\end{align}
where $r':=a^{-1}b$ and $r:=ab$. }
\end{remark}
We now compute the right hand side of Eqn.~(\ref{eigen4}) by dividing it into three partial sums: $P_1$, $P_2$ and $P_3$, where $P_1$ is the contribution of the summands with $c=2$, $P_2$ is the contribution of the summands with $c=0$, and $P_3$ is the contribution of the summands with $c=1,3$. That is, we have
$$\frac{4(q^3-1)}{q-1}S_{a,b}=P_1+P_2+P_3.$$
It is obvious that $P_1$ is equal to $0$. Next we evaluate $P_2$.
\begin{lemma}\label{lemma:c0}
We have
\[P_2=2(q+1)(1-q^3)+2q^3 \delta_{r},
\]
where
\[
\delta_{r}:=\left\{
\begin{array}{ll}
N,& \mbox{ if $\log_\omega(r) \;(\mod{N})\in T$,} \\
0,& \mbox{ if $\log_\omega(r) \; (\mod{N})\not\in T$,}
 \end{array}
\right.
\]
and $T$ is as defined in Theorem~\ref{thm:main2}.
\end{lemma}
\proof
Recall that  $T=\{\log_{\omega}(x)\,(\mod{N}):\,\Tr(x)=0\}$.
Note that $4(A\cup B)\equiv T\,(\mod{N})$ by
the definition of $A,B$
and the fact that $G(\chi_{N}^d)=q\sum_{t\in A\cup B}\chi_{N}^d(\omega^{4t})$ which follows from Theorem~\ref{thm:Yama}. We have
\begin{equation}
P_2=2\sum_{t\in A\cup B}\sum_{d=0}^{N-1}G(\chi_{N}^{-d})G(\chi_{N}^{-d})
\chi_N^d(r)
\chi_{N}^{d}(\omega^{4t}) 
=\frac{2}{q} \sum_{d=1}^{N-1}G(\chi_{N}^{-d})G(\chi_{N}^{-d})
\chi_N^d(r)
G(\chi_{N}^{d})+2(q+1).
\label{eigen5}
\end{equation}
Since $G(\chi_{N}^{-d})G(\chi_{N}^{d})=q^3$ and $G(\chi_{N}^d)=q\sum_{x\in T}\chi_{N}^d(\omega^{x})$, continuing from (\ref{eigen5}), we have
\begin{align*}
P_2=&\,2q^2\sum_{d=1}^{N-1}G(\chi_{N}^{-d})
\chi_N^d(r)+2(q+1)
=2q^3\sum_{d=1}^{N-1}\sum_{x\in T}\chi_{N}^{-d}(\omega^{x})
\chi_{N}^{d}(r)+2(q+1)\nonumber\\
=&\,2q^3\sum_{d=0}^{N-1}\sum_{x\in T}\chi_{N}^{-d}(\omega^{x})
\chi_{N}^{d}(r)+2(q+1)-2(q+1)q^3\nonumber\\
=&\,2(q+1)(1-q^3)+2q^3\cdot \left\{
\begin{array}{ll}
N,& \mbox{ if $\log_\omega(r) (\mod{N})\in T$} \\
0, & \mbox{ if $\log_\omega(r) (\mod{N})\not\in T$}
 \end{array}
\right.
\end{align*}
The conclusion of the lemma now follows.
\qed
\vspace{0.3cm}

It remains to evaluate $P_3$.

\begin{lemma}\label{lemma:c1,3}
Let $\eta$ be the quadratic character of $E$, and
define $n_{r'}=\chi_4(r')+\chi_4^3(r')+ i(\chi_4(r')-\chi_4^3(r'))$. Then
\begin{equation}
P_3=\eta(2)n_{r'}G(\eta)\Big(2N\sum_{t\in \overline{X}}\psi_E(r^{\frac{N+1}{2}}\omega^t C_0^{(2N,q^3)})
+N-q\delta_{r}'\Big),
\end{equation}
where $\delta_r'$ is defined by
\begin{align*}
\delta_r'= \left\{
\begin{array}{ll}
0,& \mbox{ if $\log_\omega(r)\; (\mod{N}) \in T$,}\\
N,& \mbox{ if $\log_\omega(r)\; (\mod{N})\not \in T$ and\;  $\log_\omega(r^{N+1})\; (\mod{2N}) \in T'$,} \\
-N,& \mbox{ if $\log_\omega(r)\; (\mod{N})\not \in T$ and \; $\log_\omega(r^{N+1}) \;(\mod{2N})\not \in T'$,}
 \end{array}
\right.
\end{align*}
where $T'=\{\log_\omega(x)\,(\mod{2N})\,|\, \Tr(x)=1\}$.
\end{lemma}
\begin{remark}\label{rem:nr}
{\em Before going into the proof of the lemma, we make the following observations.
\begin{enumerate}
\item[(i)] The $n_{r'}$ in Lemma~\ref{lemma:c1,3} can be explicitly evaluated:
\[
n_{r'}=
\left\{
\begin{array}{ll}
2& \mbox{ if $\chi_4(r')=1$ or 
$\chi_4(r')=-i$, 
}\\
-2& \mbox{ if $\chi_4(r')=-1$ or 
$\chi_4(r')=i$.
}
 \end{array}
\right.
\]
\item[(ii)] The character sum
$\sum_{t\in \overline{X}}\psi_E(r^{\frac{N+1}{2}}\omega^tC_0^{(2N,q^3)})=
\sum_{t\in \overline{X}}\psi_F(\Tr(r^{\frac{N+1}{2}} \omega^t) S)
$ in Lemma~\ref{lemma:c1,3} has been already evaluated in Theorem~\ref{thm:main2} (and Remark~\ref{tuvalue}) as:
\[
T_{\log_{\omega}(r^{\frac{N+1}{2}})}=
\left\{
\begin{array}{ll}
\frac{q-1}{2}+q\frac{-1+\eta'(2) G(\eta')}{2},& \mbox{ if $\log_{\omega} (r^\frac{N+1}{2})\;(\mod{2N})\in \overline{X}$,}\\
\frac{q-1}{2}+q\frac{-1-\eta'(2) G(\eta')}{2},& \mbox{ if  $\log_{\omega} (r^\frac{N+1}{2})\;(\mod{2N})\in \overline{X}+N$,}\\
\frac{q-1}{2},& \mbox{ if $\log_\omega(r) (\mod{N})\not \in T$ and  $\log_\omega(r^{N+1})\; (\mod{2N}) \in T'$,} \\
-\frac{q+1}{2},& \mbox{ if $\log_\omega(r) (\mod{N})\not \in T$ and  $\log_\omega(r^{N+1})\; (\mod{2N})\not \in T'$.}
 \end{array}
\right.
\]
\end{enumerate}}
\end{remark}
{\bf Proof of Lemma~\ref{lemma:c1,3}:\,} First we note that
\begin{align}
&\,P_3=n_{r'}\sum_{d=0}^{N-1}G(\chi_{4N}^{N-4d})G(\chi_{4N}^{-N-4d})
\chi_{N}^{d}(r)
\Big(\sum_{t\in A}\chi_{N}^{d}(\omega^{4t})
-\sum_{t\in B}
\chi_{N}^{d}(\omega^{4t})\Big).  \label{eigen6}
\end{align}
Applying the Hasse-Davenport product formula with $\ell=2$, $\chi=\chi_N^{-d}\chi_{4}$, and $\theta=\eta$, we have
\begin{equation}\label{HDp1}
G(\chi_N^{-d}\chi_4)G(\chi_N^{-d}\chi_4^3)=\eta(2)G(\eta)G(\chi_N^{-2d}\eta).
\end{equation}

Substituting Eqn.~(\ref{HDp1}) into Eqn.~(\ref{eigen6}), we obtain
\begin{align}
\frac{\eta(2)}{n_{r'}G(\eta)}\cdot P_3=&\,\sum_{d=0}^{N-1}G(\chi_N^{-2d}\eta)
\chi_{N}^{d}(r)
\Big(\sum_{t\in A}\chi_{N}^{2d}(\omega^{2t})
-\sum_{t\in B}
\chi_{N}^{2d}(\omega^{2t})\Big) \nonumber\\
=&\,\sum_{d=0}^{N-1}
G(\chi_N^{-2d}\eta)
\chi_{N}^{2d}(r^{\frac{N+1}{2}})\eta(r^{\frac{N+1}{2}})
\Big(\sum_{t\in A}\chi_{N}^{2d}\eta(\omega^{2t})+\sum_{t\in B}\chi_{N}^{2d}\eta(\omega^{2t+N})\Big)\nonumber\\
=&\,\sum_{t\in \overline{X}}\sum_{d=0}^{N-1}
G(\chi_N^{-2d}\eta)
\chi_N^{2d}\eta(r^{\frac{N+1}{2}})
\chi_N^{2d}\eta(\omega^{t})
\end{align}
where we have used the definition of $\overline{X}$ and the fact that $\eta(r^\frac{N+1}{2})=1$. By orthogonality of characters, we have
\begin{equation}\label{eigen8}
\frac{\eta(2)}{n_{r'}G(\eta)}\cdot P_3=\,2N\sum_{t\in \overline{X}}\psi_E(r^{\frac{N+1}{2}}\omega^t C_0^{(2N,q^3)})-\sum_{t\in \overline{X}}\sum_{d=0}^{N-1}
G(\chi_N^{-d})
\chi_{N}^{d}(r^{\frac{N+1}{2}}\omega^t).
\end{equation}
Finally, we evaluate the second sum in Eqn.~(\ref{eigen8}).
Using $G(\chi_N^{2^{-1}d})=q\sum_{t\in \overline{X}}\chi_{N}^{d}(\omega^t)$ and applying the Hasse-Davenport product formula again with $\ell=2$, $\chi=\chi_N^{-2^{-1}d}$, and
$\theta=\eta$,  we can rewrite the second sum in Eqn.~(\ref{eigen8}) as follows:
\begin{align}
\sum_{t\in \overline{X}}\sum_{d=0}^{N-1}
G(\chi_N^{-d})
\chi_{N}^{d}(r^{\frac{N+1}{2}}\omega^t)
=&\,\frac{1}{q}\sum_{d=1}^{N-1}
G(\chi_N^{-d})
G(\chi_N^{2^{-1}d})\chi_N^{d}(r^{\frac{N+1}{2}})-(q+1)\nonumber\\
=&\,\frac{G(\eta)}{q} \sum_{d=1}^{N-1}
G(\chi_N^{-2^{-1}d}\eta)\chi_N^{d}(r^{\frac{N+1}{2}})-(q+1)\nonumber\\
=&\,\frac{G(\eta)}{q} \sum_{d=0}^{N-1}
G(\chi_N^{-d}\eta)\chi_N^{d}\eta(r^{N+1})-N.\label{eigen9}
\end{align}
Let $\chi_{2N}$ be a multiplicative character of order $2N$ of $E$ and
$\eta'$ be the quadratic character of $F$.
We now use the following formula:
\[
G(\chi_N^{-d}\eta)=\Big(\sum_{x\in E: \Tr(x)=1}\chi_N^{-d}\eta(x)\Big)G(\eta'),
\]
which follows from Theorem~\ref{thm:Yama}.
Noting that $G(\eta)G(\eta')=q^2$ (c.f.~\cite[Theorem~5.18]{LN97}), we have
\begin{align*}
&\sum_{t\in \overline{X}}\sum_{d=0}^{N-1}
G(\chi_N^{-d})
\chi_{N}^{d}(r^{\frac{N+1}{2}}\omega^t)+N\\&=\,\frac{G(\eta)}{q}G(\eta')\sum_{d=0}^{N-1}
\Big(\sum_{x\in E: \Tr(x)=1}\chi_N^{-d}\eta(x)\Big)\chi_N^{d}\eta(r^{N+1})\\
&=\,q\sum_{d=0}^{2N-1}
\Big(\sum_{x\in E: \Tr(x)=1}\chi_{2N}^{-d}(x)\Big)\chi_{2N}^{d}(r^{N+1})-q\sum_{d=0}^{N-1}
\Big(\sum_{x\in E: \Tr(x)=1}\chi_{N}^{-d}(x)\Big)\chi_N^{d}(r^{N+1})\\
&=q\cdot \left\{
\begin{array}{ll}
0,& \mbox{ if $\log_\omega(r)\; (\mod{N}) \in T$,}\\
N,& \mbox{ if $\log_\omega(r)\; (\mod{N})\not \in T$ and  $\log_\omega(r^{N+1}) (\mod{2N}) \in T'$,} \\
-N,& \mbox{ if $\log_\omega(r)\; (\mod{N})\not \in T$ and  $\log_\omega(r^{N+1}) (\mod{2N})\not \in T'$.}
 \end{array}
\right.
\end{align*}
This completes the proof of the lemma.
\qed
\vspace{0.3cm}


We are now ready to complete the proof of Theorem~\ref{thm:main}. \vspace{0.1cm}

\noindent{\bf Proof of Theorem~\ref{thm:main}:}\,\,
If either $a=0$ or $b=0$,  then $\psi_{a,b}(D)=-\frac{q^2-1}{2}$ by Lemma~\ref{lemma:ab0}. If $a\neq 0$ and $b\not=0$,
by Lemmas~\ref{lemma:c0}, \ref{lemma:c1,3}, and Remark~\ref{rem:nr}, we have
\begin{align*}
\frac{4(q^3-1)}{q-1}S_{a,b}=&\,2(q+1)(1-q^3)+2q^3 \delta_r+\eta(2)n_{r'}G(\eta)(2NT_{\log_{\omega}(r^{\frac{N+1}{2}})}+N-q\delta_r')
\\
=&2(q+1)(1-q^3)+\left\{
\begin{array}{ll}
4q^3 N,& \mbox{if $\log_{\omega} (r^\frac{N+1}{2})\;(\mod{2N})\in \overline{X}$, and $n_{r'}=2$,}\\
 & \mbox{\, or  $\log_{\omega} (r^\frac{N+1}{2})\;(\mod{2N})\in \overline{X}+N$, and $n_{r'}=-2$,}\\
0, & \mbox{otherwise.}
 \end{array}
\right.
\end{align*}
Thus by Lemma~\ref{lem:abnot0} and Remark~\ref{re:transeq}, the value of $\psi_{a,b}(D)$  in this case is given by
$-\frac{q^2-1}{2}$ or $q^3-\frac{q^2-1}{2}$.
\vspace{0.3cm}

Finally, we show that $\psi_{a,b}(D)=q^3-\frac{q^2-1}{2}$  if and only if
$(b,a)\in D$.
If $(b,a)\in D$, then there are $x\in F^*$, $y\in C_0^{(q-1,q^3)}$, $z\in C_0^{(4N,q^3)}$, $u\in X$, and $v\in \{1,\omega^{N}\}$ such that
$b= xy$ and $a=xy^{-1}z u^2v$. Then,
$r:=ab=x^2z u^2v$ and $r':=a^{-1}b=yz^{-1}u^{-2}v^{-1}$, and they satisfy
\[
\log_{\omega}(r^{\frac{N+1}{2}})\,(\mod{2N})\equiv
\log_{\omega}(u^{N+1})\,(\mod{2N})\in \overline{X}+\log_{\omega}(u^N)
\]
and
\[
\chi_4(r')=\chi_4(u^{-2}v^{-1})
=\left\{
\begin{array}{ll}
1,& \mbox{ if $u$ is a square and $v=1$,} \\
-i,& \mbox{ if $u$ is a square and $v=\omega^{N}$,} \\
-1,& \mbox{ if  $u$ is a nonsquare and $v=1$,} \\
i,& \mbox{ if  $u$ is a nonsquare and $v=\omega^{N}$.}
 \end{array}
\right.
\]
If $u$ is a square,  then $\log_{\omega}(r^{\frac{N+1}{2}})\;(\mod{2N})\in \overline{X}$ and $n_{r'}=2$.
If $u$ is a nonsquare,  then $\log_{\omega}(r^{\frac{N+1}{2}})\;(\mod{2N})\in \overline{X}+N$ and $n_{r'}=-2$.
Thus, in both cases, we have $\psi_{a,b}(D)=q^3-\frac{q^2-1}{2}$. The converse also holds since the size of the dual of $D$ is $|D|$, c.f.
\cite[Theorem 3.4]{M94}. The proof of the theorem is now complete. \qed

\begin{remark}
\begin{enumerate}
\item It is clear that the subset $D$ defined in (\ref{def:D}) is disjoint from both $\{(y,0)\mid y\in E^*\}$ and $\{(0,y)\mid y\in E^*\}$. By adding one of these sets to $D$, we therefore obtain a Cameron-Liebler line class of parameter $\frac{q^2+1}{2}$.

\item We observe that $I_X=2\overline{X}\cup (2\overline{X}+N)$. Consider the element $\sigma\in P\Gamma O^+(5,q)$ defined by $\sigma((x,y))=(x^q,y^q),\,\forall (x,y)\in E\times E$. This automorphism $\sigma$ belongs to the embedded image of $P\Gamma L(4,q)$ in $P\Gamma O^+(5,q)$ induced by the Klein correspondence. We claim that $\sigma$ stabilizes our line set ${\mathcal L}$. This amounts to the fact that $D$ is invariant under the map $(x,y)\mapsto(x^q,y^q)$, which follows from the fact that $\overline{X}$ is invariant under multiplication by $q$ as proved in Remark \ref{rem_aut_pre}. Summing up, our line set ${\mathcal L}$ has an automorphism group isomorphic to $(\Z_{q^2+q+1}\times \Z_{q-1})\rtimes \Z_3$.

\end{enumerate}
\end{remark}

\section{Affine two-intersection sets}\label{sec:aff}
A set $\cK$ of points of a projective or affine plane is called a {\it set of type $(m,n)$} if every line of the plane intersects $\cK$ in $m$ or $n$ points; we assume that $m<n$, and we require both values to occur. There are many known sets of type $(m,n)$ in $\PG(2,q)$ for both even and odd $q$, e.g., a maximal arc of degree $n$ in $\PG(2,2^f)$ is a set of type $(0,n)$ with $n|2^f$, and a unital in $\PG(2,q^2)$ is  a set of type $(1,q+1)$. The situation is quite different for affine planes. When $q$ is even, let $\cK$ be a maximal arc of degree $n$ in $\PG(2,q)$ and let $\ell$ be a line of $\PG(2,q)$ such that $|\ell\cap \cK|=0$. Then $\cK$ is a set of type $(0,n)$ in $\AG(2,q)=\PG(2,q)\setminus \ell$. Since nontrivial maximal arcs do not exist in $\PG(2,q)$ when $q$ is odd, c.f.~\cite{BBM},  the construction just mentioned does not work in $\AG(2,q)$, $q$ odd. In fact, for affine planes of odd order, we only know examples of sets of type $(m,n)$ in affine planes of order $9$ and in $\AG(2,81)$. Penttila and Royle \cite{PR} classified sets of type $(3,6)$ in all affine planes of order $9$ by exhaustive computer search. In \cite{rodgers}, Rodgers developed a method to obtain new affine two-intersection sets in $\AG(2,3^{2e})$ by establishing certain tactical decompositions of the points and lines of $\PG(3,3^{2e})$ induced by Cameron-Liebler line classes with some nice properties. He was thus able to rediscover an example of sets of type $(3,6)$ in $\AG(2,9)$, and obtain a new example of affine two-intersection sets in $\AG(2,81)$. In his thesis \cite{Rthesis}, Rodgers made the following conjecture.
\begin{conj}[\cite{Rthesis}]\label{conj_aff2int}
For each integer $e\ge 1$, there exists a set of type $(\frac{1}{2}(3^{2e}-3^e), \frac{1}{2}(3^{2e}+3^e))$ in $\AG(2,3^{2e})$.
\end{conj}
Since we just established the existence of Cameron-Liebler line classes of parameters $\frac{3^{2e}-1}{2}$ in $\PG(3,3^{2e})$ for any $e\ge 1$, it is natural to ask whether Rodgers' method can be applied to these new line classes to produce affine two-intersection sets in $\AG(2,3^{2e})$. We have not been able to check whether the tactical decompositions induced by the Cameron-Liebler line classes constructed in Section 4 have the properties predicted by Rodgers. On the other hand, we are able to establish Conjecture~\ref{conj_aff2int} by a direct and explicit algebraic construction. Before doing so, we describe the model of $\AG(2, 3^{2e})$ we are going to use.
\begin{notation}\label{nota_AG} {\em Let $q=3^{2e}$ with $e\geq 1$. Write $N=q^2+q+1$, and let $\omega$ be a primitive element of $E=\F_{q^3}$.  Let $w_1$ be an element of order $N$ in $E$.  We identify the points of $\PG(2,q)$ with $\Z_N$ as follows: View $E$ as a 3-dimensional vector space over $F=\F_q$, and use $E$ as the underlying vector space of $\PG(2,q)$. We identify the projective point $\la w_1^i\ra$ with $i\in\Z_N$, $0\le i\leq N-1$. Let $J=\{i:\,i\in\Z_N,\,\Tr(w_1^i)=0\}$. We obtain the affine plane $\AG(2,q)$ from $\PG(2,q)$ by deleting $J$. The points and lines of $\AG(2,q)$ are given below:
\begin{enumerate}
\item the points: $\Z_N\setminus J$,
\item the lines: $\l_{i}=\{j\in \Z_N\setminus J : \Tr(w_1^{i+j})=0\}$, $1\leq i\leq N-1$.
\end{enumerate}}

\end{notation}
The following is our main theorem of this section.
\begin{theorem} \label{conj:affine} With notation as above,  let $w_0=\omega^{N}\in F$ and  $C_i^{(4,q)}=w_0^i\langle w_0^4\rangle$, $0\leq i\leq 3$. Then $\cK:=\{k\in \Z_N:\, \Tr(w_1^k)\in C_0^{(4,q)}\cup C_1^{(4,q)}\}\subseteq \Z_N\setminus J$  is a set of type $(\frac{1}{2}(3^{2e}-3^e), \frac{1}{2}(3^{2e}+3^e))$ in $\AG(2,3^{2e})$.
\end{theorem}
We divide the proof into a series of lemmas. The following lemma reduces the problem  to the computation of the modulus of a certain exponential sum.
\begin{lemma}\label{lem:reduce} Let $\gamma$ be any fixed element of $E\setminus F$, and let $\chi_4$ be a multiplicative character of $E^*$ of order $4$ such that
$\chi_4(\omega^N)=i$, where $i=\sqrt{-1}$.
Define $a=\overline{c}=\frac{1-i}{4}$ and
\begin{equation}\label{eqn_H}
H_{\gamma,j}:=\sum_{x\in F}\chi_4^j(1+\gamma x),\, \, j=1,3.
\end{equation}
Then the size of the set
$\{i\in \Z_N:\,\Tr(w_1^i)\in C_0^{(4,q)}\cup C_1^{(4,q)},\Tr(\gamma w_1^i)=0\}$
is equal to
\begin{equation}
\frac{q}{2}+aH_{\gamma,1}+cH_{\gamma,3}. \label{eq:finalH}
\end{equation}
\end{lemma}
\begin{proof}
The size of the set $\{i\in \Z_N:\,\Tr(w_1^i)\in C_0^{(4,q)}\cup C_1^{(4,q)},\Tr(\gamma w_1^i)=0\}$ is given by
\begin{align}
&\frac{1}{q^2}\sum_{e,f\in F}\sum_{y\in C_0^{(q-1,q^3)}}
\sum_{x\in C_0^{(4,q)}\cup C_1^{(4,q)}}\psi_E(ey)\psi_F(-ex)\psi_E(f\gamma y)
\nonumber\\
=&\frac{1}{q^2}\sum_{e\in F^\ast}\sum_{\lambda\in F}
\sum_{y\in C_0^{(q-1,q^3)}}\sum_{x\in C_0^{(4,q)}\cup C_1^{(4,q)}}
\psi_E(ye(1+\lambda\gamma))\psi_F(-ex)\label{eq:twoint12}\\
&\, \, \, \, \hspace{0.3cm} +\frac{1}{q^2}\sum_{f\in F}\sum_{y\in C_0^{(q-1,q^3)}}
\sum_{x\in C_0^{(4,q)}\cup C_1^{(4,q)}}\psi_E(yf\gamma)\label{eq:twoint2}.
\end{align}
The  sum in (\ref{eq:twoint2}) is equal to
\[
S_2:=\frac{1}{q^2}\sum_{z\in E^\ast}
\sum_{x\in C_0^{(4,q)}\cup C_1^{(4,q)}}
\psi_E(z\gamma)+\frac{N(q-1)}{2q^2}=\frac{q^2-1}{2q}.
\]
Let $\chi_{q-1}$ be a multiplicative character of order $q-1$ of $E$ and
${\chi'}_{q-1}$ be its restriction to $F$. Write $\chi_4=\chi_{q-1}^{\frac{q-1}{4}}$ and
${\chi'}_4={\chi'}_{q-1}^{\frac{q-1}{4}}$, which are multiplicative characters of order $4$ of $E$ and $F$ respectively.
Here, we can assume that $\chi_4(\omega^N)=i$.
Then, by orthogonality of characters,  the sum (\ref{eq:twoint12}), denoted by $S_1$, is computed as follows.
\begin{align}
S_1=&\frac{1}{4q^2(q-1)}\sum_{i=0}^{q-2}\sum_{j=0}^{3}
G(\chi_{q-1}^{-i})G({\chi'}_4^j)
\sum_{a\in F^\ast}\sum_{\lambda\in F}\sum_{h=0,1}
\chi_{q-1}^{i}(a(1+\lambda \gamma)){\chi'}_4^{-j}(-aw_0^h)\nonumber\\
=&\frac{1}{4q^2(q-1)}\sum_{i=0}^{q-2}\sum_{j=0}^{3}
G(\chi_{q-1}^{-i})G({\chi'}_4^j)
\sum_{\lambda\in F}\sum_{h=0,1}
\chi_{q-1}^{i}(1+\lambda \gamma){\chi'}_4^{-j}(-w_0^{h})
\sum_{a\in F^\ast}{\chi'}_{q-1}^{i-\frac{q-1}{4}j}(a). \label{eq:twoint3}
\end{align}
Here $\sum_{a\in F^\ast}{\chi'}_{q-1}^{i-\frac{q-1}{4}j}(a)=q-1$ or $0$ according as
$i\equiv \frac{q-1}{4}j\,(\mod{q-1})$ or not. Hence we have
\[
S_1=\frac{1}{4q^2}\sum_{j=0}^{3}
G(\chi_{4}^{-j})G({\chi'}_4^j)
\sum_{\lambda\in F}\sum_{h=0,1}
\chi_{4}^{j}(1+\lambda \gamma){\chi'}_4^{-j}(-w_0^{h}).
\]
Noting that  $G(\chi_{4}^{-j})G({\chi'}_4^j)=q^2$ for $j=1,2,3$, c.f. \cite[Theorem 11.6.3]{BEW97}, and ${\chi'}_4(-1)=1$, we can rewrite the above as
\begin{align*}
&\frac{1}{4}\sum_{j=0}^{3}
\sum_{\lambda\in F}\sum_{h=0,1}
\chi_{4}^{j}(1+\lambda \gamma){\chi'}_4^{-j}(w_0^{h})+\frac{1-q^2}{2q}.
\end{align*}
Therefore, the size of the set $\{i\in \Z_N:\,\Tr(w_1^i)\in C_0^{(4,q)}\cup C_1^{(4,q)},\Tr(\gamma w_1^i)=0\}$ is equal to
\begin{align*}
S_1+S_2&=\frac{1}{4}\sum_{j=0}^{3}\sum_{\lambda\in F}\sum_{h=0,1}
\chi_{4}^{j}(1+\lambda \gamma){\chi}_4^{-j}(\omega^{Nh})
\\
&=\,\frac{1}{4}\sum_{\lambda\in F}\sum_{h=0,1}1+
\frac{1}{4}\Big(\sum_{h=0,1}\chi_4^{2}(\omega^{Nh})\Big)\sum_{\lambda\in F}\chi_4^2(1+\lambda \gamma)\nonumber\\
&\, \, \, +
\frac{1}{4}\Big(\sum_{h=0,1}\chi_4^{3}(\omega^{Nh})\Big)\sum_{\lambda\in F}\chi_4(1+\lambda \gamma)+\frac{1}{4}\Big(\sum_{h=0,1}\chi_4(\omega^{Nh})\Big)\sum_{\lambda\in F}\chi_4^3(1+\lambda \gamma)\nonumber\\
&=\,\frac{q}{2}+aH_{\gamma,1}+cH_{\gamma,3}. 
\end{align*}
This completes the proof of the lemma. \qed
\end{proof}
\vspace{0.3cm}

Now, the exponential sum $H_{\gamma,j}$ can be transformed to an exponential sum over $F$ as follows:
Since 
$\chi_4(z)=\chi_4^{3}(\Norm(z))$ for $z\in E$, where $\Norm$ is the norm from $E$ to $F$, we have
\begin{align}
H_{\gamma,j}=&\chi_4^{j}(\gamma)\sum_{x \in F}\chi_4^j(\gamma^{-1}+x)
\nonumber\\
=&\chi_4^{j}(\gamma)\sum_{x \in F}\chi_4^{-j}(\Norm(\gamma^{-1}+x))\nonumber\\
=&\chi_4^{j}(\gamma)\sum_{x \in F}\chi_4^{-j}(x^3+\Tr(\gamma^{-1})x^2+\Tr(\gamma^{-1-q})x+\Norm(\gamma^{-1})).
\label{eq:cube}
\end{align}
Write
\[
f_\gamma(x):=x^3+\Tr(\gamma^{-1})x^2+\Tr(\gamma^{-1-q})x+\Norm(\gamma^{-1}).
\]
The rest of this section is devoted to showing that the exponential sum $
\sum_{x \in F}\chi_4^{-j}(f_\gamma(x))$
has modulus  $3^e$.
Note that if $|\sum_{x \in F}\chi_4^{-j}(f_\gamma(x))|=3^e$, then $\sum_{x \in F}\chi_4^{-j}(f_\gamma(x))=\alpha 3^{e}$, for a fourth root of unity $\alpha$,  since $\Z[i]$ is a unique factorization domain and $3$ is a prime in $\Z[i]$;  it  follows from Lemma~\ref{lem:reduce} that the intersection sizes of $\cK$ with the lines of $\AG(2,3^{2e})$ are $\frac{3^{2e}+ 3^{e}}{2}$ or $\frac{3^{2e}-3^{e}}{2}$.

\vspace{0.2cm}
Therefore, we are interested in the modulus of the exponential sum $\sum_{x \in F}\chi_4^{-j}(f_\gamma(x))$.  The well known theorem of Weil on multiplicative character sums implies that
$|\sum_{x\in F}
\chi_4^{-j}(f_{\gamma}(x))|\le 2\cdot 3^{e}$, c.f.~\cite[Theorem 5.41]{LN97}, which is useful but not enough.

\begin{lemma}\label{lem:affinesum}
For $j=1$ or $3$, it holds that
\begin{equation}\label{eq:HH}
\sum_{x \in F}\chi_4^{-j}(f_\gamma(x))=
\frac{-1}{3^{e}}\sum_{a\in F^\ast}\chi_4^{j}(a)
\sum_{x\in F}\psi(a\Tr(\gamma^{-1})x^2+(a\Tr(\gamma^{-1-q})+a^{1/3})x+a\Norm(\gamma^{-1}))).
\end{equation}
\end{lemma}
\proof
By orthogonality of characters, 
\begin{equation}\label{eq:quadtrans}
\sum_{x \in F}\chi_4^{-j}(f_\gamma(x))=\frac{G(\chi_4^{-j})}{3^{2e}}\sum_{a\in F^\ast}\chi_4^{j}(a)
\sum_{x\in F}\psi(af_{\gamma}(x)).
\end{equation}
Noting that $G(\chi_4^{-j})=-3^e$, c.f.~\cite[Theorem~11.6.3]{BEW97},
we have
\begin{align*}
\sum_{x \in F}\chi_4^{-j}(f_\gamma(x))=&\frac{-1}{3^{e}}\sum_{a\in F^\ast}\chi_4^{j}(a)
\sum_{x\in F}\psi(a(x^3+\Tr(\gamma^{-1})x^2+\Tr(\gamma^{-1-q})x+\Norm(\gamma^{-1})))\\
=&\frac{-1}{3^{e}}\sum_{a\in F^\ast}\chi_4^{j}(a)
\sum_{x\in F}\psi(a\Tr(\gamma^{-1})x^2+(a\Tr(\gamma^{-1-q})+a^{1/3})x+a\Norm(\gamma^{-1}))).
\end{align*}
The proof of the lemma is complete.
\qed

\vspace{0.3cm}
First we consider the case where $\Tr(\gamma^{-1})=0$ in Lemma~\ref{lem:affinesum}.
\begin{lemma}\label{lem_sqtr}
Let $\gamma\in E\setminus F$ be such that $\Tr(\gamma)=0$. Then $\Tr(\gamma^{1+q})$ is a nonzero square of $F$.
\end{lemma}
\begin{proof}
By Hilbert's Theorem 90~\cite[Theorem 2.24]{LN97}, there exists $y\in E$ such that $\gamma=y-y^q$. We can directly compute that $\Tr(\gamma^{1+q})=\Tr(y^{1+q})-\Tr(y^2)$ and
$\Tr(y)^2=\Tr(y^2)+2\Tr(y^{1+q})$. It follows that $\Tr(\gamma^{1+q})=-\Tr(y)^2$. Since $q\equiv 1\pmod{4}$, $-1$ is a square, and hence $\Tr(\gamma^{1+q})$ is a square of $F$.

Next we show that $\Tr(\gamma^{1+q})\ne 0$. The conic $\{\la x\ra:\,\Tr(x^{1+q})=0\}$ in $\PG(2,q)$ contains the point $\la 1\ra$, and the tangent line through $\la 1\ra$ is $\{\la x\ra:\,\Tr(x)=0\}$. Since $\Tr(\gamma)=0$, the point $\la \gamma\ra$ lies on this tangent line, and thus $\Tr(\gamma^{1+q})\ne 0$.\qed
\end{proof}

\begin{proposition}\label{prop:zero}
Let $\gamma$ be an element of $E\setminus F$ such that
$\Tr(\gamma^{-1})= 0$. Then, for $j=1$ or $3$,  we have \[
\Big|\sum_{x \in F}\chi_4^{-j}(f_\gamma(x))\Big|=3^e.
\]
\end{proposition}
\proof
By Lemma~\ref{lem:affinesum}, we need to show that
\[
\Big|\sum_{a\in F^\ast}\chi_4^{j}(a)
\sum_{x\in F}\psi((a\Tr(\gamma^{-1-q})+a^{1/3})x+a\Norm(\gamma^{-1})))\Big|=3^{2e}.
\]
Since  $\Tr(\gamma^{-1-q})$ is a nonzero square
of $F$ by Lemma~\ref{lem_sqtr} and $-1$ is a square in $F$ (in fact, $-1$ is a fourth power in $F$), there is an element $t\in F$ such that $t^2=-\Tr(\gamma^{-1-q})$. Then
$a\Tr(\gamma^{-1-q})+a^{1/3}=0$ if and only if  $a=0,\pm t^{-3}$. Hence
\begin{align*}
&\Big|\sum_{a\in F^\ast}\chi_4^{j}(a)
\sum_{x\in F}\psi((a\Tr(\gamma^{-1-q})+a^{1/3})x+a\Norm(\gamma^{-1})))\Big|\\
=&\,
3^{2e}|\psi(t^{-3}\Norm(\gamma^{-1}))+\psi(-t^{-3}\Norm(\gamma^{-1}))|.
\end{align*}
Noting that
\begin{equation}\label{eq:cubeset}
\Big\{t^3\Big(\frac{x^3}{t^3}-\frac{x}{t}+\frac{\Norm(\gamma^{-1})}{t^3}\Big):\,x\in F\Big\}=\Big\{y:\,\Tr_{q/3}\Big(\frac{y}{t^3}\Big)=\Tr_{q/3}\Big(\frac{\Norm(\gamma^{-1})}{t^3}\Big)\Big\},
\end{equation}
we have
$\Tr_{q/3}(t^{-3}\Norm(\gamma^{-1}))\not=0$; otherwise
the set (\ref{eq:cubeset}) contains zero, and $f_{\gamma}(x)=0$ for some $x\in F$, which is impossible since $f_\gamma(x)=\Norm(\gamma^{-1}+x)$.
Hence,  we have $\psi(t^{-3}\Norm(\gamma^{-1}))+\psi(-t^{-3}\Norm(\gamma^{-1}))=-1$, which completes the proof.  
\qed
\vspace{0.3cm}
Next, we consider the case where $\Tr(\gamma^{-1})\not=0$ in Lemma~\ref{lem:affinesum}.
\begin{lemma}\label{lem:transGK}
Let $\gamma$ be an element of $E\setminus F$ such that
$\Tr(\gamma^{-1})\not= 0$. Then, for $j=1$ or $3$, there exists an element $z\in F$ such that
\begin{equation}\label{eq:Kloo}
\Big|\sum_{x \in F}\chi_4^{-j}(f_\gamma(x))\Big|=\Big|\sum_{a\in F^\ast}\chi_4^{j}(a)\psi(za+a^{-1})\Big|.
\end{equation}
\end{lemma}
\proof
If $a\Tr(\gamma^{-1})\not=0$, by \cite[Theorem~5.33]{LN97}, we have
\begin{align*}
& \sum_{x\in F}\psi(a\Tr(\gamma^{-1})x^2+(a\Tr(\gamma^{-1-q})+a^{1/3})x+a\Norm(\gamma^{-1})))\\
&=G(\eta')\psi(a\Norm(\gamma^{-1})-(a\Tr(\gamma^{-1-q})+a^{1/3})^2a^{-1}\Tr(\gamma^{-1})^{-1})\eta'(a\Tr(\gamma^{-1})),
\end{align*}
where $\eta'$ is the quadratic character of $F$.
Write $a_0=\Norm(\gamma^{-1})^{1/3}$, $a_1=\Tr(\gamma^{-1-q})$, and
$a_2=\Tr(\gamma^{-1})^{-1}$.
Then, by Lemma~\ref{lem:affinesum} and $G(\eta')=\pm 3^e$ (c.f. \cite[Theorem~5.15]{LN97}), 
we obtain
\begin{align*}
& \Big|\sum_{x \in F}\chi_4^{-j}(f_\gamma(x))\Big|\\
=&\,\Big|\sum_{a\in F^\ast}\chi_4^{j}(a)
\psi(a\Norm(\gamma^{-1})-(a\Tr(\gamma^{-1-q})+a^{1/3})^2a^{-1}\Tr(\gamma^{-1})^{-1})\eta'(a\Tr(\gamma^{-1}))\Big|\\
=&\,\Big|\sum_{a\in F^\ast}\chi_4^{3j}\eta'(a)
\psi(a^3\Norm(\gamma^{-1})-(a^3a_1+a)^2a^{-3}a_2)\Big|\\
=&\,\Big|\sum_{a\in F^\ast}\chi_4^{j}(a)
\psi(a_0a-a_2a^{-3}(a^6a_1^2-a^4a_1+a^2))\Big|\\
=&\,\Big|\sum_{a\in F^\ast}\chi_4^{j}(a)
\psi(ya-a_2a^{-1})\Big|=\Big|\sum_{a\in F^\ast}\chi_4^{j}(a)
\psi(za+a^{-1})\Big|
\end{align*}
where $y=a_0-(a_1^2a_2)^{1/3}+a_1a_2$ and $z=-ya_2$.
\qed

\vspace{0.3cm}
The exponential sum $K_{j,z}:=\sum_{a\in F^\ast}\chi_4^{j}(a)\psi(za+a^{-1})$ appearing in the right-hand side of Eqn.~(\ref{eq:Kloo})  is a {\it generalized
Kloosterman sum} \cite[p.~265]{LN97}. It is clear that
$\sum_{x \in F}\chi_4^{-j}(f_\gamma(x))=\alpha K_{j,z}$ for a fourth root
of unity $\alpha$ by the proof above.
If $z=0$ in Eqn.~(\ref{eq:Kloo}),
then the sum  $K_{j,z}$ is just a Gauss sum, and hence $\sum_{x\in F}\chi_4^{-j}(f_\gamma(x))$ has modulus $3^e$.

Now, we assume that $z\not=0$. Again by orthogonality of characters,  the  sum $K_{j,z}$
can be expressed in terms of Gauss sums as follows:
\begin{align*}
K_{j,z}=&\frac{1}{(q-1)^2}\sum_{h,i=0}^{q-1}G(\chi^{-h})G(\chi^{-i})\sum_{a\in F^\ast}
\chi_4^{j}(a)\chi^{h}(za)\chi^{-i}(a),
\end{align*}
where $\chi$ is a multiplicative character of order $q-1$ of $F$. Here, we can assume that $\chi^{\frac{q-1}{4}}=\chi_4$. Since the inner sum
$\sum_{a\in F^\ast}
\chi_4^{j}(a)\chi^{h}(a)\chi^{-i}(a)=q-1$ or $0$ according to $i\equiv h+\frac{q-1}{4}j\,(\mod{q-1})$ or not, we see that
\[
K_{j,z}=\frac{1}{q-1}\sum_{h=0}^{q-1}G(\chi^{-h})G(\chi^{-h-\frac{q-1}{4}j})
\chi^{h}(z).
\]
In the following lemma, we show that $3^e\,|\,G(\chi^{-h})G(\chi^{-h-\frac{q-1}{4}j})$ for every $h=0,1,\ldots,q-2$, which  implies that
$3^e$ divides $K_{j,z}$.
\begin{lemma}\label{lem:GaussDiv}
For any $h=0,1,\ldots,q-2$ and $j=1$ or $3$, we have
\[
3^e\,|\,G(\chi^{-h})G(\chi^{-h-\frac{q-1}{4}j}),
\]
where $\chi$ is a multiplicative character of order $q-1$ of $F$.
\end{lemma}
\proof
We will only prove the lemma in the case where $j=1$. The case where $j=3$ is similar. If $h=0$, $G(\chi^{-h})G(\chi^{-h-\frac{q-1}{4}})=G(\chi^{-\frac{q-1}{4}})=-3^{e}$ by \cite[Theorem~11.6.3]{BEW97}. Similarly, if $h=\frac{3(q-1)}{4}$, then $G(\chi^{-h})G(\chi^{-h-\frac{q-1}{4}})=G(\chi^{-\frac{3(q-1)}{4}})=-3^{e}$. Thus, we will assume that $h\not=0$ or $\frac{3(q-1)}{4}$ below.

By Theorem~\ref{stick}, it is enough to show that
\[
s(h)+s(h+\frac{q-1}{4}) \ge 2e
\]
for all $h=1,\ldots, q-2$, where $s(x)$ is the sum of the $3$-adic digits of the reduction of $x$ modulo $q-1$.

Define $a\equiv h+\frac{q-1}{4}\,(\mod{q-1})$.  For any $x\in \Z_{q-1}$, $x\neq 0$, write $x=\sum_{i=0}^{2e-1}x_i3^i=x_{2e-1}x_{2e-2}\cdots x_{0}$ with $x_i\in \{0,1,2\}$, where the subscripts are taken modulo $2e$.
Note that $\frac{q-1}{4}=\sum_{i=0}^{e-1}2\cdot 3^{2i}=0202\cdots 02$. We now use the modular $p$-ary add-with-carry algorithm
described in \cite[Theorem 4.1]{HHKWX}, which says that  there is a unique carry sequence $c=	c_{2e-1}c_{2e-2}\cdots c_0$ with $c_i\in \{0,1\}$  such that
for all $0\le i\le 2e-1$
\begin{equation}\label{eq:modularalgo}
a_i+3c_i=c_{i-1}+h_i+(1+(-1)^i).
\end{equation}
It follows that
\begin{align*}
s(h)+s(h+\frac{q-1}{4})=\sum_{i=0}^{2e-1}h_i+\sum_{i=0}^{2e-1}a_i
=2\sum_{i=0}^{2e-1}h_i-2\sum_{i=0}^{2e-1}c_i+2e.
\end{align*}
Thus, if  $\sum_{i=0}^{2e-1}h_i\ge \sum_{i=0}^{2e-1}c_i$ is shown, then we obtain the assertion of this lemma. We now prove a stronger inequality, namely,
$h_{2j+1}+h_{2j+2}\ge c_{2j+1}+c_{2j+2}$ for $0\le j\le e-1$, from which it follows that  $\sum_{i=0}^{2e-1}h_i\ge \sum_{i=0}^{2e-1}c_i$.

If either of $h_{2j+1}$ or $h_{2j+2}$ is greater than or equal to $2$, then
the stronger inequality clearly holds since $c_i\in \{0,1\}$. So we assume that
both $h_{2j+1}$ and $h_{2j+2}$ are less than $2$.

By Eqn.~(\ref{eq:modularalgo}), we have
\[
a_{2j+1}+3c_{2j+1}=c_{2j}+h_{2j+1}\, \mbox{ and }\,  a_{2j+2}+3c_{2j+2}=c_{2j+1}+h_{2j+2}+2.
\]
If $c_{2j+1}=1$, then  $a_{2j+1}=0$, $c_{2j}=1$, and $h_{2j+1}=2$, which implies that $h_{2j+1}+h_{2j+2}\ge c_{2j+1}+c_{2j+2}$. If $c_{2j+1}=0$,
then $a_{2j+2}+3c_{2j+2}=h_{2j+2}+2$. In the case where $h_{2j+2}=0$, we have  $a_{2j+2}=2$
and $c_{2j+2}=0$. In the case where $h_{2j+2}=1$, we have $a_{2j+2}=0$
and $c_{2j+2}=1$. In both cases, we have  $h_{2j+1}+h_{2j+2}\ge c_{2j+1}+c_{2j+2}$. The proof of the lemma is complete.
\qed

\vspace{0.3cm}
Now recall that  $\sum_{x\in F}\chi_4^{-j}(f_\gamma(x))=\alpha K_{j,z}$ for a fourth root unity $\alpha$.
Since now we have shown that $3^e$ divides
$K_{j,z}$, we can write
$\sum_{x\in F}\chi_4^{-j}(f_\gamma(x))=3^e(a+bi)$
for some $a,b\in \Z$.
Since
$|\sum_{x\in F}\chi_4^{-j}(f_\gamma(x))|\le 2\cdot 3^e$,
we see that $(a,b)$ is equal to one of
\[
(a,b)=(0,0),(0,\pm 1), (0,\pm 2), (\pm 1,\pm 1), (\pm 1,0),(\pm 2,0).
\]
If $(a,b)=(0,0), (0,\pm 2), (\pm 1,\pm 1)$,  or $(\pm 2,0)$,  then the intersection size stated in (\ref{eq:finalH}) is not integral, a contradiction. Therefore, it must be that $(a,b)=(0,\pm 1)$ or
$(\pm 1,0)$, i.e.,  $|\sum_{x\in F}\chi_4^{-j}(f_\gamma(x))|=3^e$. Summing up, we have proved the following.
\begin{proposition}\label{prop:notzero}
Let $\gamma$ be an element of $E\setminus F$ such that
$\Tr(\gamma^{-1})\not=0$. Then, for $j=1$ or $3$,  it holds that \[
\Big|\sum_{x \in F}\chi_4^{-j}(f_\gamma(x))\Big|=3^e.
\]
\end{proposition}
By Lemma~\ref{lem:reduce}  together with Proposition~\ref{prop:zero} and $\ref{prop:notzero}$, we obtain  the assertion of
Theorem~\ref{conj:affine}.

\begin{remark} It is well known (c.f. \cite{CK}) that projective two-intersection sets are equivalent to certain strongly regular Cayley graphs (hence certain two-class association schemes). It is natural to ask what combinatorial objects are behind affine two-intersection sets. We give an answer in this remark. The notation here is almost the same as that in Notation~\ref{nota_AG}. The only difference is that we simply identify the points of $\PG(2,q)$ with $i\in \Z_N$, and define $J=\{j:\,j\in\Z_N,\,\Tr(\omega^j)=0\}$, and $l_i=\{j\in \Z_N\setminus J : \, \Tr(\omega^{i+j})=0\}$, $1\leq i\leq N-1$, where $\Tr$ is the relative trace from $E$ to $F$.
Assume that $X\subseteq \Z_N\setminus J$ is a set of type
$(m,n)$ in $\AG(2,q)$. Then there exists a set $Y\subseteq \Z_N\setminus \{0\}$ such that
\begin{equation}\label{eq:intersec}
|X\cap l_{i}|=\begin{cases}m,\quad &\textup{if $i \in Y$},\\
n,\quad & \textup{if $i \in \Z_N\setminus (Y \cup \{0\})$}.
\end{cases}
\end{equation}
Define
\[
D_0:=\{0\}, \, D_1:=\bigcup_{i\in J}C_i^{(N,q^3)}, \,
D_2:=\bigcup_{i\in X}C_i^{(N,q^3)}, \, D_3:=\bigcup_{i\in \Z_N\setminus (J\cup X)}C_i^{(N,q^3)}.
\]
Let $\psi$ be the canonical additive character of $E$.
Then, by  \eqref{eq:intersec}, the values  of $\psi(\omega^a D_2):=\sum_{x\in D_2}\psi(\omega^ax)$, $a=0,1,\ldots,N-1$, can be computed as follows:
\begin{align*}
\sum_{x \in D_2}\psi(\omega^a x)=&\, \sum_{i\in X}\sum_{\lambda\in \F_q^\ast}
\psi(\omega^{a+i} \lambda )=-|X|+q(|X\cap (J-a)|)\\
=&\, -|X|+\begin{cases}qm,\quad &\textup{if $a \in Y$},\\
qn,\quad & \textup{if $a \in \Z_N\setminus (Y \cup \{0\})$},\\
0,\quad & \textup{if $a =0$}.
\end{cases}
\end{align*}
On the other hand, it is clear that
\begin{align*}
\sum_{x \in D_1}\psi(\omega^a x)=
\begin{cases}q^{2}-1,\quad &\textup{if $a=0$},\\
-1,\quad & \textup{otherwise}.
\end{cases}
\end{align*}
Furthermore, $\sum_{x \in D_3}\psi(\omega^a x)$, $a=0,1,\ldots,N-1$,  can be computed as
$-1-\sum_{x \in D_1\cup D_2}\psi(\omega^a x)$. Thus, for each $i=1,2,3$, the character values $\psi(\omega^a D_i)$, $a=0,1,\ldots,N-1$, are constant according to $a=0$, $a\in Y$, or $a\in \Z_N\setminus (Y\cup \{0\})$. (In the language of association schemes, the Cayley graphs $Cay(E,D_i)$, $i=0,1,2,3$, form a three-class
association scheme on $E$. See, e.g., \cite[Theorem 10.1]{God}.) Thus, as an immediate consequence of our result on affine two-intersection sets,
we obtain a  three-class association scheme on $\F_{3^{6e}}$. Conversely, starting from a three-class association scheme defined by the three Cayley graphs $Cay(E,D_i)$, $1\leq i\leq 3$, we obtain an affine
two-intersection set in $\AG(2,q)$. Summing up, we have the following proposition.
\begin{proposition}\label{prop:affcombin}
With notation as above, a subset $X\subseteq \Z_N\setminus J$ is an affine two-intersectionin set in $\AG(2,q)$ if and only if there exists a subset $Y\subseteq \Z_N\setminus \{0\}$ such that  
for each $i=1,2,3$, the character values $\psi(\omega^a D_i)$, $a=0,1,\ldots,N-1$, are constant according to  $a=0$, $a\in Y$, or $a\in \Z_N\setminus (Y\cup \{0\})$ (or, equivalently, the Cayley graphs $Cay(E,D_i)$, $i=0,1,2,3$, form a three-class
association scheme on $E$).
\end{proposition}
Note that a result similar to Proposition~\ref{prop:affcombin} holds for affine two-intersection sets in  $\AG(s,q)$ for $s\ge 3$. We omit the detailed statement.
\end{remark}

\section*{Concluding Remarks}
In this paper, we constructed an infinite family of Cameron-Liebler line classes in $\PG(3,q)$ with parameter $x=\frac{q^2-1}{2}$, where $q\equiv 5$ or $9\pmod {12}$. Furthermore we constructed the first infinite family of sets of type $(m,n)$ in the affine plane $\AG(2, q)$, where $q$ is an even power of $3$. It would be interesting to come up with a general construction of Cameron-Liebler line classes in $\PG(3,q)$ when $q$ is even since there are some known examples in this case in the thesis \cite{Rthesis} of Rodgers.

We close this paper by referring the reader to a paper \cite{BDMR14} by De Beule, Demeyer, Metsch, and Rodgers. Immediately after we finished a draft of this manuscript, we became aware that De Beule, Demeyer, Metsch and Rodgers \cite{BDMR14} also obtained the same result on Cameron-Liebler line classes with parameter $x=\frac{q^2-1}{2}$ at almost the same time. The approaches for proving the main result are comparable but different enough to justify that we write two separate papers;  our approach is more algebraic and the approach taken by De Beule, Demeyer, Metsch and Rodgers is more geometric.  The two teams of authors discussed this matter with each other, and decided to submit their  papers separately.

\end{document}